\documentclass[final,3p]{elsarticle}
 \usepackage[latin1]{inputenc}
 \usepackage{graphics}
 \usepackage{graphicx}
 \usepackage{epsfig}
\usepackage{amssymb}
 \usepackage{amsthm}
 \usepackage{lineno}
 \usepackage{amsmath}
   \numberwithin{equation}{section}
\usepackage{mathrsfs}

\newtheorem{thm}{Theorem}[section]

\newtheorem{lem}[thm]{Lemma}

\newtheorem{defn}[thm]{Definition}

\biboptions{sort&compress,square}
\begin{document}
\begin{frontmatter}
\author[rvt1]{Jian Wang}
\ead{wangj484@nenu.edu.cn}
\author[rvt2]{Yong Wang\corref{cor2}}
\ead{wangy581@nenu.edu.cn}
\author[rvt2]{Tong Wu}
\cortext[cor2]{Corresponding author.}
\address[rvt1]{School of Science, Tianjin University of Technology and Education, Tianjin, 300222, P.R.China}
\address[rvt2]{School of Mathematics and Statistics, Northeast Normal University,
Changchun, 130024, P.R.China}

\title{ Dirac operators with torsion, spectral Einstein functionals\\ and the noncommutative residue}
\begin{abstract}

Recently Dabrowski etc. \cite{DL} obtained the metric and Einstein
functionals by two vector fields and Laplace-type operators over vector bundles, giving an interesting example of
the spinor connection and square of the Dirac operator. Pf$\ddot{a}$ffle and Stephan \cite{PS1} considered
orthogonal connections with arbitrary torsion on compact Riemannian manifolds and computed the spectral
action.
 Motivated by the spectral functionals and Dirac operators with torsion,
  we give some new spectral functionals which is the extension of spectral functionals to the noncommutative realm with torsion,
 and we relate them to  the noncommutative residue for manifolds with boundary.
  Our method of producing these spectral functionals is the noncommutative residue and Dirac operators with torsion.
\end{abstract}
\end{frontmatter}
\section{Introduction}
\label{1}
An eminent spectral scheme that generates geometric objects on
manifolds such as residue, scalar curvature, and other scalar combinations of curvature tensors
is the small-time asymptotic expansion of the (localised) trace
of heat kernel\cite{PBG,FGV}. The theory has very rich structures both in physics and mathematics.
In the recent paper\cite{DL}, Dabrowski etc. defined bilinear functionals of vector fields and differential forms,
the densities of which yield the  metric and Einstein spectral functionals on even-dimensional Riemannian manifolds,
and they obtained certain
values or residues of the (localised) zeta function of the Laplacian  arising from
the Mellin transform and the coefficients of this expansion.

Let $E$ be a finite-dimensional complex vector bundle over a closed compact manifold $M$
of dimension $n$, the noncommutative residue of a pseudo-differential operator
$P\in\Psi DO(E)$ can be defined by
 \begin{equation}
res(P):=(2\pi)^{-n}\int_{S^{*}M}\mathrm{Tr}(\sigma_{-n}^{P}(x,\xi))\mathrm{d}x \mathrm{d}\xi,
\end{equation}
where $S^{*}M\subset T^{*}M$ denotes the co-sphere bundle on $M$ and
$\sigma_{-n}^{P}$ is the component of order $-n$ of the complete symbol
 \begin{equation}
\sigma^{P}:=\sum_{i}\sigma_{i}^{P}
\end{equation}
of $P$, cf. \cite{Ac,Wo,Wo1,Gu},
and the linear functional $res: \Psi DO(E)\rightarrow \mathbb{C }$
is in fact the unique trace (up to multiplication
by constants) on the algebra of pseudo-differential operators $\Psi DO(E)$.
In \cite{Co1}, Connes  computed a conformal four-dimensional
 Polyakov action analogy using the noncommutative residue.
Connes  proved that the noncommutative residue on a compact manifold $M$ coincided with Dixmier's trace on pseudodifferential
operators of order -dim$M$\cite{Co2}.
More precisely, Connes made a challenging observation that
the Wodzicki residue of the inverse square of the Dirac operator yields the
Einstein-Hilbert action of general relativity\cite{Co2}\cite{Co3}. Kastler\cite{Ka} gave a brute-force proof of this theorem, and Kalau and Walze\cite{KW} proved
this theorem in the normal coordinates system simultaneously, which is called the Kastler-Kalau-Walze theorem now.
Based on the theory of the noncommutative reside  introduced by Wodzicki, Fedosov etc.\cite{FGLS} constructed a noncommutative
residue on the algebra of classical elements in Boutet de Monvel's calculus on a compact manifold with boundary of dimension $n>2$.
With elliptic pseudodifferential operators and  noncommutative
residue, it is natural way to study the Kastler-Kalau-Walze type theorem and
operator-theoretic explanation of the gravitational action for manifolds with boundary.
For Dirac operators and signature operators,
Wang computed the noncommutative residue and proved Kastler-Kalau-Walze type theorem for manifolds with boundary\cite{Wa1,Wa3,Wa4}.

 Earlier Jean-Michel Bismut \cite{JMB} proved a local index theorem for Dirac operators on a
Riemannian manifold $M$ associated with connections on $TM$ which have non zero
torsion.
In \cite{AT}, Ackermann and Tolksdorf proved a generalized version of the well-known Lichnerowicz formula for the square of the
most general Dirac operator with torsion $D_{T}$ on an even-dimensional spin manifold associated to a metric connection with torsion.
In \cite{PS}, Pf$\ddot{a}$ffle and Stephan considered compact Riemannian spin manifolds without boundary equipped with orthogonal connections,
and investigated the induced Dirac operators. Moreover, Pf$\ddot{a}$ffle and Stephan considered
orthogonal connections with arbitrary torsion on compact Riemannian manifolds, and for the induced
Dirac operators, twisted Dirac operators and Dirac operators of Chamseddine-Connes type they computed the spectral
action\cite{PS1}. In \cite{WWY}, we computed the lower
dimensional  volume $\widetilde{{\rm Wres}}[\pi^+(D^{*}_{T})^{-p_{1}}\circ\pi^+D_{T}^{-p_{2}}]$ and got  a  Kastler-Kalau-Walze
type theorems associated with Dirac operators with torsion on compact manifolds with  boundary.
The purpose of this paper is to generalize the results in \cite{DL},\cite{Wa3}, \cite{PS} and get spectral functionals
 associated with Dirac operators with torsion on compact manifolds with  boundary. For lower dimensional compact Riemannian manifolds
  with  boundary, we compute the lower dimensional  residue of $\widetilde{\nabla}_{X}\widetilde{\nabla}_{Y}D_{T}^{-4}$ and
   get the Kastler-Kalau-Walze theorems.

\section{Spectral functionals for Dirac operator with Torsion}

 In this section we consider an $n$-dimensional oriented Riemannian manifold $(M, g^{M})$ equipped
with some spin structure. The Levi-Civita connection
$\nabla: \Gamma(TM)\rightarrow \Gamma(T^{*}M\otimes TM)$ on $M$ induces a connection
$\nabla^{S}: \Gamma(S)\rightarrow \Gamma(T^{*}M\otimes S).$ By adding a additional torsion term $t\in\Omega^{1}(M,End TM)$ we
obtain a new covariant derivative
 \begin{equation}
\widetilde{\nabla}:=\nabla+t
\end{equation}
on the tangent bundle $TM$. Since $t$ is really a one-form on $M$ with values in the bundle of skew endomorphism $Sk(TM)$ in \cite{GHV},
$\nabla$ is in fact compatible with the Riemannian metric $g$ and therefore also induces a connection $\widetilde{\nabla}^{S}:=\nabla^{S}+T$
on the spinor bundle. Here $T\in\Omega^{1}(M, End S)$ denotes the `lifted' torsion term $t\in\Omega^{1}(M, End TM)$.

 Next, we will briefly discuss the construction of this connection. Again, we write $\tilde{\nabla}_{X}Y=\nabla_{X}Y+A(X,Y)$
 with the Levi-Civita connection $\nabla$.
For any $X \in T_{p}M$ the endomorphism $A(X,\cdot)$ is skew-adjoint and hence it is an element of
$\mathfrak{so}(T_{p}M) $, we can express it as
\begin{equation}
A(X,\cdot)=\sum_{i<j}\alpha_{ij}E_{i}\wedge E_{j}.
\end{equation}
 Here $E_{i}\wedge E_{j}$ is meant as the endomorphism of $T_{p}M$ defined by $E_{i}\wedge E_{j}$.
 For any $X \in T_{p}M$ one determines the coefficients in (2.2) by
\begin{equation}
\alpha_{ij}=\langle A(X,E_{i}), E_{j}\rangle=A_{XE_{i} E_{j}}.
\end{equation}
 Each  $E_{i}\wedge E_{j}$  lifts to $\frac{1}{2}E_{i}\cdot E_{j}$
in $spin(n)$, and the spinor connection induced by $\widetilde{\nabla}$ is locally
given by
 \begin{equation}
\widetilde{\nabla}_{X}\psi=\nabla_{X}\psi+\frac{1}{2}\sum_{i<j}\alpha_{ij}E_{i}\cdot E_{j}\psi=
\nabla_{X}\psi+\frac{1}{2}\sum_{i<j}A_{XE_{i} E_{j}}E_{i}\cdot E_{j}\psi.
\end{equation}
The connection given by (2.4) is compatible with the metric on spinors and with
Clifford multiplication. Then, the Dirac operator associated to the spinor connection from (2.4) is defined as
\begin{align}
D_{T}\psi&=\sum_{i=1}^{n}E_{i}\tilde{\nabla}_{E_{i}}\psi=D\psi+\frac{1}{2}\sum_{i=1}^{n}\sum_{j<k}A_{E_{i} E_{j}E_{k}}E_{i}\cdot E_{j}
\cdot E_{k}\psi\nonumber\\
&=D\psi+\frac{1}{4}\sum_{i,j,k=1}^{n}A_{E_{i} E_{j}E_{k}}E_{i}\cdot E_{j}\cdot E_{k}\psi
\end{align}
where $D$ is the Dirac operator induced by the Levi-Civita connection and $``\cdot"$ is the
Clifford multiplication. Let $c(e)$ be the Clifford operators acting on $S(TM)$, satisfying
\begin{align}
c(e_{i})c(e_{j})+c(e_{j})c(e_{i})=-2\langle e_{i}, e_{j}  \rangle.
\end{align}
For any orthogonal connection $\tilde{\nabla}$  on the tangent bundle of $M$ there exist a unique
vector field $V$, we have
 \begin{equation}
D_{T}\psi=D\psi+\frac{3}{2}T\cdot\psi-\frac{n-1}{2}V\cdot\psi,~~D_{T}^{*}\psi=D\psi+\frac{3}{2}T\cdot\psi+\frac{n-1}{2}V\cdot\psi,
\end{equation}
where Clifford multiplication by any 3-form is self-adjoint.
As the Clifford multiplication by the vector field $V$ is skew-adjoint we get
that $D_{T}$ is symmetric with respect to the natural $L^{2}$-scalar product on spinors if and only
if the vectorial component of the torsion vanishes, $V\equiv 0$. Note that the Cartan type torsion $S$ does not contribute to the Dirac operator
$D_{T}$. As $D_{T}^{*}D_{T}$ is a generalized Laplacian,  one has  the following Lichnerowicz formula.
\begin{lem}\cite{PS}
For the Dirac operator $D_{T}$  associated to the orthogonal connection $\widetilde{\nabla}$, we have
\begin{align}
D_{T}^{*}D_{T}\psi=&\Delta\psi+\frac{1}{4}R^{g}\psi+\frac{3}{2}dT\cdot\psi-\frac{3}{4}\parallel T\parallel^{2}\psi\nonumber\\
          & +\frac{n-1}{2}div^{g}(V)\psi+(\frac{n-1}{2})^{2}(2-n)|V|^{2}\psi\nonumber\\
          & +3(n-1)(T\cdot V\cdot\psi+(V_{\rfloor}T)\cdot\psi),
\end{align}
for any spinor field $\psi$, where $\Delta$ is the Laplacian associated to the connection
 \begin{equation}
\widetilde{\nabla}_{X}\psi=\nabla_{X}\psi+\frac{3}{2}(X_{\rfloor}T)\cdot\psi-\frac{n-1}{2}V\cdot X\cdot\psi-\frac{n-1}{2}\langle V, X\rangle\psi.
\end{equation}
\end{lem}

The following lemma of Dabrowski etc.'s Einstein functional play a key role in our proof
of the Einstein functional with torsion.
Let $V$, $W$ be a pair of vector fields on a compact
Riemannian manifold $M$, of dimension $n = 2m$. Using the Laplace operator $\Delta^{-1}_{T}=(D_{T}^{*}D_{T})^{-1}=\Delta+E $
acting on sections of
 a vector bundle $S(TM)$ of rank $2^{m}$ which
 may contain both some nontrivial connections and torsion,
 the spectral functionals over vector fields defined by
\begin{lem}\cite{DL}
The Einstein functional equal to
 \begin{equation}
Wres\big(\widetilde{\nabla}_{V}\widetilde{\nabla}_{W}\Delta^{-m}_{T}\big)=\frac{\upsilon_{n-1}}{6}2^{m}\int_{M}G(V,W)vol_{g}
 +\frac{\upsilon_{n-1}}{2}\int_{M}F(V,W)vol_{g}+\frac{1}{2}\int_{M}(\mathrm{Tr}E)g(V,W)vol_{g},
\end{equation}
where $G(V,W)$ denotes the Einstein tensor evaluated on the two vector fields, $F(V,W)=Tr(V_{a}W_{b}F_{ab})$ and
$F_{ab}$ is the curvature tensor of the connection $T$, $\mathrm{Tr}E$ denotes the trace of $E$ and $\upsilon_{n-1}=\frac{2\pi^{m}}{\Gamma(m)}$.
\end{lem}
The aim of this section is to prove the following.
\begin{thm}
For the Laplace (type) operator with torsion $\Delta_{T}$, the Einstein functional equal to
\begin{align}
Wres\big(\widetilde{\nabla}_{V}\widetilde{\nabla}_{W}\Delta^{-m}_{T,E}\big)
=&\frac{2^{m+1}\pi^{m}}{6\Gamma(m)}\int_{M}\big(Ric(V,W)-\frac{1}{2}sg(V,W)\big) vol_{g}\nonumber\\
&+\int_{M}2^{m-1}
\big( -\frac{1}{4}R^{g}-\frac{3}{2}div^{g}(V)+\frac{3}{2}\parallel T\parallel^{2}+\frac{9}{2}\parallel V\parallel^{2}\big)g(V,W) vol_{g},
\end{align}
where $\int_{M}div(V) dVol_{M}=-\int_{\partial_{M}}g(n,V)dVol_{\partial_{M}}$,
$R_{g}$ denotes the curvature tensor and $s$ is the scalar curvature.
\end{thm}

\begin{proof}
Let $T(X,\cdot,\cdot)=\sum_{ 1\leq i<j\leq n}T(X,e_{i},e_{j})e_{i}^{*}\wedge e_{j}^{*}$ and
$c(T(X,\cdot,\cdot))=\sum_{ 1\leq i<j\leq n}T(X,e_{i},e_{j})c(e_{i})c(e_{j})$, repeated application of (2.11) yields that
 \begin{align}
\widetilde{\nabla}_{X}\psi=&\nabla_{X}^{S(TM)}\psi+\frac{3}{2}T(X,\cdot,\cdot)\psi-\frac{n-1}{2}c(V)c(X)\psi
-\frac{n-1}{2}\langle V, X  \rangle\psi\nonumber\\
=&X\psi+\sigma(X)\psi+\frac{3}{2}\sum_{ 1\leq i<j\leq n}T(X,e_{i},e_{j})c(e_{i})c(e_{j})\nonumber\\
& -\frac{n-1}{2}c(V)c(X)\psi-\frac{n-1}{2}\langle V, X  \rangle\psi\nonumber\\
=&X\psi+\overline{A}(X)\psi,
\end{align}
where
 \begin{equation}
\sigma(X)=-\frac{1}{4}\sum_{s,t} \omega_{s,t}(X) c(e_s)c(e_t).
\end{equation}
Let $V=\sum_{a=1}^{n}V^{a}e_{a}$,$W=\sum_{b=1}^{n}W^{b}e_{b}$,
in view of that
 \begin{equation}
F(V,W)=Tr(V_{a}W_{b}F_{ab})=\sum_{a,b=1}^{n}V^{a}W^{b}Tr^{S(TM)}(F_{e_{a},e_{b}}),
\end{equation}
we obtain
 \begin{align}
F_{e_{a},e_{b}}=&(e_{a}+\overline{A}(e_{a}))(e_{b}+\overline{A}(e_{b}))-(e_{b}+\overline{A}(e_{b}))(e_{a}+\overline{A}(e_{a}))
-([e_{a},e_{a}]+\overline{A}([e_{a},e_{b}]))
\nonumber\\
=&e_{a}\circ \overline{A}(e_{b})+\overline{A}(e_{a})\circ e_{b}+\overline{A}(e_{a})A(e_{b})-e_{b}\circ \overline{A}(e_{a})
-\overline{A}(e_{b})\circ e_{a}\nonumber\\
& -\overline{A}(e_{b})\overline{A}(e_{a})-\overline{A}([e_{a},e_{b}])\nonumber\\
=&\overline{A}(e_{b})e_{a}+e_{a}(\overline{A}(e_{b}))+\overline{A}(e_{a})\circ e_{b}+\overline{A}(e_{a})\overline{A}(e_{b})
-\overline{A}(e_{a})\circ e_{b}-e_{b}(\overline{A}(e_{a}))\nonumber\\
& -\overline{A}(e_{b})e_{a}-\overline{A}(e_{b})\overline{A}(e_{a})-\overline{A}([e_{a},e_{b}])\nonumber\\
=&e_{a}(\overline{A}(e_{b}))-e_{b}(\overline{A}(e_{a}))+\overline{A}(e_{a})\overline{A}(e_{b})
-\overline{A}(e_{b})\overline{A}(e_{a})-\overline{A}([e_{a},e_{b}]).
\end{align}
Also, straightforward computations yield
 \begin{align}
\mathrm{Tr}^{S(TM)}\big(e_{a}(\overline{A}(e_{b}))\big)
=&\mathrm{Tr}^{S(TM)}\Big[e_{a}\Big(-\frac{1}{4}\sum_{s,t} \omega_{s,t}(e_{b}) c(e_s)c(e_t)
+\frac{3}{2}\sum_{ 1\leq i<j\leq n}T(e_{b},e_{i},e_{j})c(e_{i})c(e_{j})\nonumber\\
&-\frac{n-1}{2}c(V)c(e_{b})-\frac{n-1}{2}\langle V, e_{b} \rangle\Big)\Big]\nonumber\\
=&\mathrm{Tr}^{S(TM)}\Big[-\frac{1}{4}\sum_{s,t} e_{a}(\omega_{s,t}(e_{b})) c(e_s)c(e_t)
+\frac{3}{2}\sum_{ 1\leq i<j\leq n}e_{a}(T(e_{b},e_{i},e_{j}))c(e_{i})c(e_{j})\nonumber\\
&-\frac{n-1}{2}\sum_{k=1}^{n}e_{a}(V^{k})c(e_{k})c(e_{b})-\frac{n-1}{2}e_{a}(V^{b})\Big]\nonumber\\
=&\mathrm{Tr}^{S(TM)}\Big[-\frac{1}{8}\sum_{s=t} R^{TM}_{abst}c(e_s)c(e_t)
+\frac{3}{2}\sum_{ 1\leq i<j\leq n}e_{a}(T(e_{b},e_{i},e_{j}))c(e_{i})c(e_{j})\nonumber\\
&-\frac{n-1}{2}\sum_{k=1}^{n}e_{a}(c(V^{k}))c(e_{k})c(e_{b})-\frac{n-1}{2}e_{a}(V^{b})\Big]\nonumber\\
=&-\frac{1}{8}\sum_{s=t} R^{TM}_{abst}(-\delta_{st})Tr^{S(TM)}(Id)-\frac{n-1}{2}e_{a}(V^{b})Tr^{S(TM)}(Id)\nonumber\\
&-\frac{n-1}{2}\sum_{k=1}^{n}e_{a}(c(V^{k}))(-\delta_{kb})Tr^{S(TM)}(Id)\nonumber\\
=&0,
\end{align}
where
 \begin{equation}
e_{a}(\omega_{s,t}(e_{b}))(x_0)=\frac{\partial}{\partial y^{a}}(\omega_{s,t}(e_{b}))(x_0)
=\frac{\partial}{\partial y^{a}}(H_{b,s,t})(x_0)=\frac{1}{2}R^{TM}_{abst}(x_0),
\end{equation}
where we take the normal coordinate about $x_0$, it follows that
\begin{align}
&\mathrm{Tr}^{S(TM)}\big(\overline{A}(e_{a})\overline{A}(e_{b})-\overline{A}(e_{b})\overline{A}(e_{a})\big)\nonumber\\
=&\mathrm{Tr}^{S(TM)}\Big[ \Big(\frac{3}{2}\sum_{ 1\leq i<j\leq n}T(e_{a},e_{i},e_{j})c(e_{i})c(e_{j})
-\frac{n-1}{2}c(V)c(e_{a})-\frac{n-1}{2}\langle V, e_{a} \rangle\Big)\nonumber\\
&\times \Big(\frac{3}{2}\sum_{ 1\leq i<j\leq n}T(e_{b},e_{i},e_{j})c(e_{i})c(e_{j})
-\frac{n-1}{2}c(V)c(e_{b})-\frac{n-1}{2}\langle V, e_{b} \rangle\Big)\Big]\nonumber\\
&-\mathrm{Tr}^{S(TM)}\Big[\Big(\frac{3}{2}\sum_{ 1\leq i<j\leq n}T(e_{b},e_{i},e_{j})c(e_{i})c(e_{j})
-\frac{n-1}{2}c(V)c(e_{b})-\frac{n-1}{2}\langle V, e_{b} \rangle\Big)
\nonumber\\
&\times \Big(\frac{3}{2}\sum_{ 1\leq i<j\leq n}T(e_{a},e_{i},e_{j})c(e_{i})c(e_{j})
-\frac{n-1}{2}c(V)c(e_{a})-\frac{n-1}{2}\langle V, e_{a} \rangle\Big))\Big]\nonumber\\
=&0,
\end{align}
and
\begin{align}
&~~~~\mathrm{Tr}^{S(TM)}\big(\overline{A}([e_{a},e_{b}])\big)(x_{0})\nonumber\\
=&\mathrm{Tr}^{S(TM)} \Big(\sigma([e_{a},e_{b}])+\frac{3}{2}\sum_{ 1\leq i<j\leq n}T([e_{a},e_{b}],e_{i},e_{j})c(e_{i})c(e_{j})
-\frac{n-1}{2}c(V)c([e_{a},e_{b}])\nonumber\\
&-\frac{n-1}{2}\langle V, [e_{a},e_{b}] \rangle\Big)\Big)(x_{0})\nonumber\\
=&0.
\end{align}
Let $\Delta_{T}=\Delta+E$, by (53) in \cite{PS1}, we have
 \begin{equation}
E=\Big(-\frac{1}{4}R^{g}-\frac{3}{2}div^{g}(V)+\frac{3}{2}\parallel T\parallel^{2}+\frac{9}{2}\parallel V\parallel^{2}\Big)Id_{S(TM)}
-\frac{3}{2}dT-9T\cdot V-9V_{\rfloor}T,
\end{equation}
Now, what is left is to show that
\begin{align}
\mathrm{Tr}^{S(TM)}(E)=&2^{\frac{n}{2}}\Big(-\frac{1}{4}R^{g}-\frac{3}{2}div^{g}(V)+\frac{3}{2}\parallel T\parallel^{2}
+\frac{9}{2}\parallel V\parallel^{2}\Big)\nonumber\\
&-Tr^{S(TM)}(\frac{3}{2}dT)-9Tr^{S(TM)}(T\cdot V)-9Tr^{S(TM)}(V_{\rfloor}T)\nonumber\\
=&2^{\frac{n}{2}}\Big(-\frac{1}{4}R^{g}-\frac{3}{2}div^{g}(V)+\frac{3}{2}\parallel T\parallel^{2}
+\frac{9}{2}\parallel V\parallel^{2}\Big)\nonumber\\
&-\frac{3}{2}Tr^{S(TM)}\big(\sum_{ 1\leq \alpha<j<k<l\leq n}dT(e_{\alpha},e_{j},e_{k},e_{l})c(e_{\alpha})c(e_{j})c(e_{k})c(e_{l})\big)\nonumber\\
&-9 Tr^{S(TM)}\big(\sum_{ 1\leq \alpha<j<k\leq n}T(e_{\alpha},e_{j},e_{k})c(e_{\alpha})c(e_{j})c(e_{k})c(V)\big)\nonumber\\
&-9Tr^{S(TM)}\big(\sum_{ 1\leq i<j\leq n}T(V,e_{i},e_{j})c(e_{i})c(e_{j})\big)\nonumber\\
=&2^{\frac{n}{2}}\Big(-\frac{1}{4}R^{g}-\frac{3}{2}div^{g}(V)+\frac{3}{2}\parallel T\parallel^{2}
+\frac{9}{2}\parallel V\parallel^{2}\Big)\nonumber\\
&-9\sum_{l=1}^{n}V_{l}\sum_{ 1\leq \alpha<j<k\leq n}T(e_{\alpha},e_{j},e_{k})Tr^{S(TM)}\big(c(e_{\alpha})c(e_{j})c(e_{k})c(e_{l})\big)\nonumber\\
=&2^{\frac{n}{2}}\Big(-\frac{1}{4}R^{g}-\frac{3}{2}div^{g}(V)+\frac{3}{2}\parallel T\parallel^{2}
+\frac{9}{2}\parallel V\parallel^{2}\Big).
\end{align}
Summing up (2.18)-(2.21) leads to the desired equality (2.11), and the proof of
the Theorem is complete.
\end{proof}

\section{Noncommutative residue for manifold with boundary}

In this section, to define the noncommutative residue for Dirac operator with torsion,
 some basic facts and formulae about Boutet de Monvel's calculus can be find in Sec.2 in \cite{Wa1}.
Let $M$ be an n-dimensional compact oriented manifold with boundary $\partial M$.
Some basic facts and formulae about Boutet de Monvel's calculus are recalled as follows.

Let $$ F:L^2({\bf R}_t)\rightarrow L^2({\bf R}_v);~F(u)(v)=\int e^{-ivt}u(t)\texttt{d}t$$ denote the Fourier transformation and
$\varphi(\overline{{\bf R}^+}) =r^+\varphi({\bf R})$ (similarly define $\varphi(\overline{{\bf R}^-}$)), where $\varphi({\bf R})$
denotes the Schwartz space and
  \begin{equation}
r^{+}:C^\infty ({\bf R})\rightarrow C^\infty (\overline{{\bf R}^+});~ f\rightarrow f|\overline{{\bf R}^+};~
 \overline{{\bf R}^+}=\{x\geq0;x\in {\bf R}\}.
\end{equation}
We define $H^+=F(\varphi(\overline{{\bf R}^+}));~ H^-_0=F(\varphi(\overline{{\bf R}^-}))$ which are orthogonal to each other. We have the following
 property: $h\in H^+~(H^-_0)$ iff $h\in C^\infty({\bf R})$ which has an analytic extension to the lower (upper) complex
half-plane $\{{\rm Im}\xi<0\}~(\{{\rm Im}\xi>0\})$ such that for all nonnegative integer $l$,
 \begin{equation}
\frac{\texttt{d}^{l}h}{\texttt{d}\xi^l}(\xi)\sim\sum^{\infty}_{k=1}\frac{\texttt{d}^l}{\texttt{d}\xi^l}(\frac{c_k}{\xi^k})
\end{equation}
as $|\xi|\rightarrow +\infty,{\rm Im}\xi\leq0~({\rm Im}\xi\geq0)$.

 Let $H'$ be the space of all polynomials and $H^-=H^-_0\bigoplus H';~H=H^+\bigoplus H^-.$ Denote by $\pi^+~(\pi^-)$ respectively the
 projection on $H^+~(H^-)$. For calculations, we take $H=\widetilde H=\{$rational functions having no poles on the real axis$\}$ ($\tilde{H}$
 is a dense set in the topology of $H$). Then on $\tilde{H}$,
 \begin{equation}
\pi^+h(\xi_0)=\frac{1}{2\pi i}\lim_{u\rightarrow 0^{-}}\int_{\Gamma^+}\frac{h(\xi)}{\xi_0+iu-\xi}\texttt{d}\xi,
\end{equation}
where $\Gamma^+$ is a Jordan close curve included ${\rm Im}\xi>0$ surrounding all the singularities of $h$ in the upper half-plane and
$\xi_0\in {\bf R}$. Similarly, define $\pi^{'}$ on $\tilde{H}$,
 \begin{equation}
\pi'h=\frac{1}{2\pi}\int_{\Gamma^+}h(\xi)\texttt{d}\xi.
\end{equation}
So, $\pi'(H^-)=0$. For $h\in H\bigcap L^1(R)$, $\pi'h=\frac{1}{2\pi}\int_{R}h(v)\texttt{d}v$ and for $h\in H^+\bigcap L^1(R)$, $\pi'h=0$.
Denote by $\mathcal{B}$ Boutet de Monvel's algebra (for details, see Section 2 of \cite{Wa1}).

An operator of order $m\in {\bf Z}$ and type $d$ is a matrix
$$A=\left(\begin{array}{lcr}
  \pi^+P+G  & K  \\
   T  &  S
\end{array}\right):
\begin{array}{cc}
\   C^{\infty}(X,E_1)\\
 \   \bigoplus\\
 \   C^{\infty}(\partial{X},F_1)
\end{array}
\longrightarrow
\begin{array}{cc}
\   C^{\infty}(X,E_2)\\
\   \bigoplus\\
 \   C^{\infty}(\partial{X},F_2)
\end{array}.
$$
where $X$ is a manifold with boundary $\partial X$ and
$E_1,E_2~(F_1,F_2)$ are vector bundles over $X~(\partial X
)$.~Here,~$P:C^{\infty}_0(\Omega,\overline {E_1})\rightarrow
C^{\infty}(\Omega,\overline {E_2})$ is a classical
pseudodifferential operator of order $m$ on $\Omega$, where
$\Omega$ is an open neighborhood of $X$ and
$\overline{E_i}|X=E_i~(i=1,2)$. $P$ has an extension:
$~{\cal{E'}}(\Omega,\overline {E_1})\rightarrow
{\cal{D'}}(\Omega,\overline {E_2})$, where
${\cal{E'}}(\Omega,\overline {E_1})~({\cal{D'}}(\Omega,\overline
{E_2}))$ is the dual space of $C^{\infty}(\Omega,\overline
{E_1})~(C^{\infty}_0(\Omega,\overline {E_2}))$. Let
$e^+:C^{\infty}(X,{E_1})\rightarrow{\cal{E'}}(\Omega,\overline
{E_1})$ denote extension by zero from $X$ to $\Omega$ and
$r^+:{\cal{D'}}(\Omega,\overline{E_2})\rightarrow
{\cal{D'}}(\Omega, {E_2})$ denote the restriction from $\Omega$ to
$X$, then define
$$\pi^+P=r^+Pe^+:C^{\infty}(X,{E_1})\rightarrow {\cal{D'}}(\Omega,
{E_2}).$$
In addition, $P$ is supposed to have the
transmission property; this means that, for all $j,k,\alpha$, the
homogeneous component $p_j$ of order $j$ in the asymptotic
expansion of the
symbol $p$ of $P$ in local coordinates near the boundary satisfies:
$$\partial^k_{x_n}\partial^\alpha_{\xi'}p_j(x',0,0,+1)=
(-1)^{j-|\alpha|}\partial^k_{x_n}\partial^\alpha_{\xi'}p_j(x',0,0,-1),$$
then $\pi^+P:C^{\infty}(X,{E_1})\rightarrow C^{\infty}(X,{E_2})$
by Section 2.1 of \cite{Wa1}.

Let $M$ be a compact manifold with boundary $\partial M$. We assume that the metric $g^{M}$ on $M$ has
the following form near the boundary
 \begin{equation}
 g^{M}=\frac{1}{h(x_{n})}g^{\partial M}+\texttt{d}x _{n}^{2} ,
\end{equation}
where $g^{\partial M}$ is the metric on $\partial M$. Let $U\subset
M$ be a collar neighborhood of $\partial M$ which is diffeomorphic $\partial M\times [0,1)$. By the definition of $h(x_n)\in C^{\infty}([0,1))$
and $h(x_n)>0$, there exists $\tilde{h}\in C^{\infty}((-\varepsilon,1))$ such that $\tilde{h}|_{[0,1)}=h$ and $\tilde{h}>0$ for some
sufficiently small $\varepsilon>0$. Then there exists a metric $\hat{g}$ on $\hat{M}=M\bigcup_{\partial M}\partial M\times
(-\varepsilon,0]$ which has the form on $U\bigcup_{\partial M}\partial M\times (-\varepsilon,0 ]$
 \begin{equation}
\hat{g}=\frac{1}{\tilde{h}(x_{n})}g^{\partial M}+\texttt{d}x _{n}^{2} ,
\end{equation}
such that $\hat{g}|_{M}=g$.
We fix a metric $\hat{g}$ on the $\hat{M}$ such that $\hat{g}|_{M}=g$.
Now we recall the main theorem in \cite{FGLS}.
\begin{thm}\label{th:32}{\bf(Fedosov-Golse-Leichtnam-Schrohe)}
 Let $X$ and $\partial X$ be connected, ${\rm dim}X=n\geq3$,
 $A=\left(\begin{array}{lcr}\pi^+P+G &   K \\
T &  S    \end{array}\right)$ $\in \mathcal{B}$ , and denote by $p$, $b$ and $s$ the local symbols of $P,G$ and $S$ respectively.
 Define:
 \begin{align}
{\rm{\widetilde{Wres}}}(A)=&\int_X\int_{\bf S}{\mathrm{Tr}}_E\left[p_{-n}(x,\xi)\right]\sigma(\xi)dx \nonumber\\
&+2\pi\int_ {\partial X}\int_{\bf S'}\left\{{\mathrm{Tr}}_E\left[({\mathrm{Tr}}b_{-n})(x',\xi')\right]+{\mathrm{Tr}}
_F\left[s_{1-n}(x',\xi')\right]\right\}\sigma(\xi')dx',
\end{align}
Then~~ a) ${\rm \widetilde{Wres}}([A,B])=0 $, for any
$A,B\in\mathcal{B}$;~~ b) It is a unique continuous trace on
$\mathcal{B}/\mathcal{B}^{-\infty}$.
\end{thm}
Let $p_{1},p_{2}$ be nonnegative integers and $p_{1}+p_{2}\leq n$,
 denote by $\sigma_{l}(A)$ the $l$-order symbol of an operator $A$,
 an application of (3.5) and (3.6) in \cite{Wa1} shows that
\begin{defn} Lower-dimensional volumes of spin manifolds with boundary with torsion are defined by
   \begin{equation}\label{}
   Vol_{n}^{\{p_{1},p_{2}\}}M:=\widetilde{Wres}[\pi^{+}(\widetilde{\nabla}_{X}\widetilde{\nabla}_{Y}(D_{T}^{*}D_{T})^{- p_{1}})
    \circ\pi^{+}((D_{T}^{*}D_{T})^{-1})^{p_{2}}].
\end{equation}
where $\pi^{+}(\widetilde{\nabla}_{X}\widetilde{\nabla}_{Y}(D_{T}^{*}D_{T})^{- p_{1}})$, $\pi^{+}((D_{T}^{*}D_{T})^{-1})^{p_{2}}$ are
 elements in Boutet de Monvel's algebra\cite{Wa3}.
\end{defn}
 For Dirac operators
 with torsion $\widetilde{\nabla}_{X}\widetilde{\nabla}_{Y}(D_{T}^{*}D_{T})^{-1}$ and $(D_{T}^{*}D_{T})^{-1}$,
 denote by $\sigma_{l}(A)$ the $l$-order symbol of an operator A. An application of (2.1.4) in \cite{Wa1} shows that
\begin{align}
&\widetilde{Wres}[\pi^{+}(\widetilde{\nabla}_{X}\widetilde{\nabla}_{Y}(D_{T}^{*}D_{T})^{-1})^{p_{1}}
\circ\pi^{+}(D_{T}^{*}D_{T})^{-p_{2}}]\nonumber\\
&=\int_{M}\int_{|\xi|=1}\mathrm{Tr}_{S(TM)}
  [\sigma_{-n}(\widetilde{\nabla}_{X}\widetilde{\nabla}_{Y}(D_{T}^{*}D_{T})^{-p_{1}}
  \circ (D_{T}^{*}D_{T})^{-p_{2}})]\sigma(\xi)\texttt{d}x+\int_{\partial M}\Phi,
\end{align}
where
 \begin{align}
\Phi=&\int_{|\xi'|=1}\int_{-\infty}^{+\infty}\sum_{j,k=0}^{\infty}\sum \frac{(-i)^{|\alpha|+j+k+\ell}}{\alpha!(j+k+1)!}
\mathrm{Tr}_{S(TM)}[\partial_{x_{n}}^{j}\partial_{\xi'}^{\alpha}\partial_{\xi_{n}}^{k}\sigma_{r}^{+}
((\widetilde{\nabla}_{X}\widetilde{\nabla}_{Y}(D_{T}^{*}D_{T})^{-p_{1}})(x',0,\xi',\xi_{n})\nonumber\\
&\times\partial_{x'}^{\alpha}\partial_{\xi_{n}}^{j+1}\partial_{x_{n}}^{k}\sigma_{l}((D_{T}^{*}D_{T})^{-p_{2}})(x',0,\xi',\xi_{n})]
\texttt{d}\xi_{n}\sigma(\xi')\texttt{d}x' ,
\end{align}
and the sum is taken over $r-k+|\alpha|+\ell-j-1=-n,r\leq-p_{1},\ell\leq-p_{2}$.

For Dirac operators
 with torsion  $\widetilde{\nabla}_{X}\widetilde{\nabla}_{Y}D_{T}^{-1}$ and $(D_{T}^{*}D_{T}D_{T}^{*})^{-1}$,
 similarly we have
\begin{align}
&\widetilde{Wres}[\pi^{+}(\widetilde{\nabla}_{X}\widetilde{\nabla}_{Y}(D_{T})^{-1})^{p_{1}}
\circ\pi^{+}( D_{T}^{*}D_{T}D_{T}^{*})^{-p_{2}}]\nonumber\\
&=\int_{M}\int_{|\xi|=1}\mathrm{Tr}_{S(TM)}
  [\sigma_{-n}(\widetilde{\nabla}_{X}\widetilde{\nabla}_{Y}(D_{T})^{-p_{1}}
  \circ (D_{T}^{*}D_{T}D_{T}^{*})^{-p_{2}})]\sigma(\xi)\texttt{d}x+\int_{\partial M}\widetilde{\Phi},
\end{align}
where
 \begin{align}
\widetilde{\Phi}=&\int_{|\xi'|=1}\int_{-\infty}^{+\infty}\sum_{j,k=0}^{\infty}\sum \frac{(-i)^{|\alpha|+j+k+\ell}}{\alpha!(j+k+1)!}
\mathrm{Tr}_{S(TM)}[\partial_{x_{n}}^{j}\partial_{\xi'}^{\alpha}\partial_{\xi_{n}}^{k}\sigma_{r}^{+}
((\widetilde{\nabla}_{X}\widetilde{\nabla}_{Y}(D_{T})^{-p_{1}})(x',0,\xi',\xi_{n})\nonumber\\
&\times\partial_{x'}^{\alpha}\partial_{\xi_{n}}^{j+1}\partial_{x_{n}}^{k}\sigma_{l}((D_{T}^{*}D_{T}D_{T}^{*})^{-p_{2}})(x',0,\xi',\xi_{n})]
\texttt{d}\xi_{n}\sigma(\xi')\texttt{d}x' ,
\end{align}
and the sum is taken over $r-k+|\alpha|+\ell-j-1=-n,r\leq-p_{1},\ell\leq-p_{2}$.

 \section{Residue for Dirac operators
 with torsion $\widetilde{\nabla}_{X}\widetilde{\nabla}_{Y}(D_{T}^{*}D_{T})^{-1}$ and $(D_{T}^{*}D_{T})^{-1}$ }

In this section, we compute the lower dimensional volume for 4-dimension compact manifolds with boundary and get a
Kastler-Kalau-Walze type formula in this case.
 We will consider $D_{T}^{*}D_{T}$(since $D_{T}$ is not self-adjoint in general).
Since $[\sigma_{-4}(\widetilde{\nabla}_{X}\widetilde{\nabla}_{Y}(D_{T}^{*}D_{T})^{-1}
  \circ (D_{T}^{*}D_{T})^{-1})]|_{M}$ has
the same expression as $[\sigma_{-4}(\widetilde{\nabla}_{X}\widetilde{\nabla}_{Y}(D_{T}^{*}D_{T})^{-1}
  \circ (D_{T}^{*}D_{T})^{-1})]|_{M}$ in the case of manifolds without boundary,
so locally we can use Theorem 2.4 to compute the first term.
\begin{thm}
 Let M be a 4-dimensional compact manifold without boundary and $\widetilde{\nabla}$ be an orthogonal
connection with torsion. Then we get the volumes  associated to $\widetilde{\nabla}_{X}\widetilde{\nabla}_{Y}(D_{F}^{*}D_{F})^{-1}$
and $D_{T}^{*}D_{T}$ on compact manifolds without boundary
 \begin{align}
&Wres[\sigma_{-4}(\widetilde{\nabla}_{X}\widetilde{\nabla}_{Y}(D_{T}^{*}D_{T})^{-1}
  \circ (D_{T}^{*}D_{T})^{-1})]\nonumber\\
=&\frac{4\pi^{2}}{3}\int_{M}\Big(Ric(X,Y)-\frac{1}{2}sg(X,Y)\Big) vol_{g}\nonumber\\
&+\int_{M}
\Big( -\frac{1}{2}R^{g}-3\mathrm{div}^{g}(X)+3\parallel T\parallel^{2}+9\parallel X\parallel^{2}\Big)g(X,Y) vol_{g},
\end{align}
where $R_{g}$ denotes the curvature tensor and $s$ is the scalar curvature.
\end{thm}

So we only need to compute $\int_{\partial M}\Phi$.
 Recall the definition of the Dirac operator D in \cite{Y}. Denote by $\sigma_{l}(A)$ the $l$-order symbol of
an operator A. In the local coordinates $\{x_{i}; 1\leq i\leq n\}$ and the fixed orthonormal frame
$\{\widetilde{e_{1}},\cdots, \widetilde{e_{n}}\}$, the connection matrix $(\omega_{s,t})$ is defined by
\begin{equation}
\widetilde{\nabla}(\widetilde{e_{1}},\cdots, \widetilde{e_{n}})=(\widetilde{e_{1}},\cdots, \widetilde{e_{n}})(\omega_{s,t}).
\end{equation}
The Dirac operator
\begin{equation}
D=\sum_{i=1}^{n}c(\widetilde{e_{i}})[\widetilde{e_{i}}-\frac{1}{4}\sum_{s,t}\omega_{s,t}(\widetilde{e_{i}})c(\widetilde{e_{s}})c(\widetilde{e_{t}})],
\end{equation}
where $c(\widetilde{e_{i}})$ denotes the Clifford action. Then
\begin{align}
D_{T}=&\sum_{i=1}^{n}c(\widetilde{e_{i}})[\widetilde{e_{i}}-\frac{1}{4}\sum_{s,t}\omega_{s,t}(\widetilde{e_{i}})c(\widetilde{e_{s}})
c(\widetilde{e_{t}})]
  +\frac{1}{4}\sum_{i\neq s\neq t}A_{ist}
  c(\widetilde{e_{i}})c(\widetilde{e_{s}})c(\widetilde{e_{t}})  \nonumber\\
  &+\frac{1}{4}\sum_{i, s, t}[-A_{iit}c(\widetilde{e_{t}})
  +A_{isi}c(\widetilde{e_{s}}) -A_{iss}c(\widetilde{e_{i}})+2A_{iii}c(\widetilde{e_{i}})], \\
D_{T}^{*}=&\sum_{i=1}^{n}c(\widetilde{e_{i}})[\widetilde{e_{i}}-\frac{1}{4}\sum_{s,t}\omega_{s,t}(\widetilde{e_{i}})c(\widetilde{e_{s}})
c(\widetilde{e_{t}})]
  +\frac{1}{4}\sum_{i\neq s\neq t}A_{ist}
  c(\widetilde{e_{i}})c(\widetilde{e_{s}})c(\widetilde{e_{t}})  \nonumber\\
  &-\frac{1}{4}\sum_{i, s, t}[-A_{iit}c(\widetilde{e_{t}})
  +A_{isi}c(\widetilde{e_{s}}) -A_{iss}c(\widetilde{e_{i}})+2A_{iii}c(\widetilde{e_{i}})],
  \end{align}
and
\begin{align}
\sigma_{1}(D_{T})=&\sigma_{1}(D_{T}^{*})=\sqrt{-1}c(\xi);\\
\sigma_{0}(D_{T})=&-\frac{1}{4}\sum_{i,s,t}\omega_{s,t}(\widetilde{e_{i}})c(\widetilde{e_{s}})c(\widetilde{e_{t}})
+\frac{1}{4}\sum_{i\neq s\neq t}A_{ist}
  c(\widetilde{e_{i}})c(\widetilde{e_{s}})c(\widetilde{e_{t}})  \nonumber\\
  &+\frac{1}{4}\sum_{i, s, t}[-A_{iit}c(\widetilde{e_{t}})
  +A_{isi}c(\widetilde{e_{s}}) -A_{iss}c(\widetilde{e_{i}})+2A_{iii}c(\widetilde{e_{i}})], \\
\sigma_{0}(D_{T}^{*})=&-\frac{1}{4}\sum_{i,s,t}\omega_{s,t}(\widetilde{e_{i}})c(\widetilde{e_{s}})c(\widetilde{e_{t}})
 +\frac{1}{4}\sum_{i\neq s\neq t}A_{ist}
  c(\widetilde{e_{i}})c(\widetilde{e_{s}})c(\widetilde{e_{t}})  \nonumber\\
  &-\frac{1}{4}\sum_{i, s, t}[-A_{iit}c(\widetilde{e_{t}})
  +A_{isi}c(\widetilde{e_{s}}) -A_{iss}c(\widetilde{e_{i}})+2A_{iii}c(\widetilde{e_{i}})].
\end{align}
We define $\nabla_X^{S(TM)}:=X+\frac{1}{4}\sum_{ij}\langle\nabla_X^L{e_i},e_j\rangle c(e_i)c(e_j)$, which is a spin connection. Set
\begin{equation}
A(X)=\frac{1}{4}\Sigma_{ij}\langle\nabla_X^L{e_i},e_j\rangle c(e_i)c(e_j).
\end{equation}
 Let $\widetilde{\nabla}_{X}=X+A(X)+\overline{T}(X,\cdot,\cdot)$,
  $\overline{T}(X,\cdot,\cdot)=\frac{3}{2}\sum_{ 1\leq i<j\leq n}T(X,e_{i},e_{j})c(e_{i})c (e_{j})-\frac{1}{2}c(V)c(X)
  -\frac{1}{2}\langle V, X \rangle$ and
 $\widetilde{\nabla}_{Y}=Y+A(Y)+\overline{T}(Y,\cdot,\cdot)$, by (2.11), we obtain
\begin{align}
\widetilde{\nabla}_{X}\widetilde{\nabla}_{Y}&=(X+A(X)
+\overline{T}(X,\cdot,\cdot))(Y+A(Y)+\overline{T}(Y,\cdot,\cdot))\nonumber\\
 &=XY+X[A(Y)]+A(Y)X+A(X)Y+A(X)A(Y)+\overline{T}(X,\cdot)Y
 +\overline{T}(X,\cdot,\cdot)A(Y)\nonumber\\
 &~~~~+X[\overline{T}(Y,\cdot,\cdot)]+\overline{T}(Y,\cdot,\cdot)X+A(X)\overline{T}(X,\cdot,\cdot)+\overline{T}(Y,\cdot,\cdot)\overline{T}(Y,\cdot,\cdot),
\end{align}
where
$X=\Sigma_{j=1}^nX_j\partial_{x_j}, Y=\Sigma_{l=1}^nY_l\partial_{x_l}$.

 Let $g^{ij}=g(dx_{i},dx_{j})$, $\xi=\sum_{k}\xi_{j}dx_{j}$
  and $\nabla^L_{\partial_{i}}\partial_{j}=\sum_{k}\Gamma_{ij}^{k}\partial_{k}$,
   we get
\begin{align}
&\sigma_{i}=-\frac{1}{4}\sum_{s,t}\omega_{s,t}
(e_i)c(e_s)c(e_t)
;~~~\xi^{j}=g^{ij}\xi_{i};~~~~\Gamma^{k}=g^{ij}\Gamma_{ij}^{k};~~~~\sigma^{j}=g^{ij}\sigma_{i}.
\end{align}
Then we have the following lemmas.
\begin{lem}\label{lem3} The following identities hold:
\begin{align}
 \sigma_{0}(\widetilde{\nabla}_{X}\widetilde{\nabla}_{Y})=&X[A(Y)]+A(X)A(Y)
+\overline{T}(X,\cdot,\cdot)A(Y)\nonumber\\
&+X[\overline{T}(Y,\cdot,\cdot)]+A(X)\overline{T}(X,\cdot,\cdot)+\overline{T}(Y,\cdot,\cdot)\overline{T}(Y,\cdot,\cdot);\\
\sigma_{1}(\widetilde{\nabla}_{X}\widetilde{\nabla}_{Y})
=&\sqrt{-1}\sum_{j,l=1}^nX_j\frac{\partial Y_l }{\partial x_j }\sqrt{-1}\xi_l
+\sqrt{-1}\sum_jA(Y)X_j\xi_j+\sqrt{-1}\sum_lA(Y)Y_l\xi_l\nonumber\\
&+\sum_j  \overline{T}(X,\cdot)Y_j\sqrt{-1}\xi_j+\sum_j \overline{T}(Y,\cdot)X_j\sqrt{-1} \xi_j;\\
\sigma_{2}(\widetilde{\nabla}_{X}\widetilde{\nabla}_{Y})=&-\sum_{j,l=1}^nX_jY_l\xi_j\xi_l.
\end{align}
\end{lem}
Hence by Lemma 2.1 in \cite{Wa3}, we have
 \begin{lem}\label{le:31}
The symbol of the Dirac operator
\begin{align}
\sigma_{-1}(D_{T}^{-1})&=\sigma_{-1}((D_{T}^{*})^{-1})=\frac{\sqrt{-1}c(\xi)}{|\xi|^{2}}; \\
\sigma_{-2}(D_{T}^{-1})&=\frac{c(\xi)\sigma_{0}(D_{T})c(\xi)}{|\xi|^{4}}+\frac{c(\xi)}{|\xi|^{6}}\sum_{j}c(\texttt{d}x_{j})
\Big[\partial_{x_{j}}(c(\xi))|\xi|^{2}-c(\xi)\partial_{x_{j}}(|\xi|^{2})\Big]; \\
\sigma_{-2}((D_{T}^{*})^{-1})&=\frac{c(\xi)\sigma_{0}(D_{T}^{*})c(\xi)}{|\xi|^{4}}+\frac{c(\xi)}{|\xi|^{6}}\sum_{j}c(\texttt{d}x_{j})
\Big[\partial_{x_{j}}(c(\xi))|\xi|^{2}-c(\xi)\partial_{x_{j}}(|\xi|^{2})\Big].
\end{align}
\end{lem}
Let $u=\frac{1}{4}\sum_{i\neq s\neq t}A_{ist}
  c(\widetilde{e_{i}})c(\widetilde{e_{s}})c(\widetilde{e_{t}}) $,
  $v=\frac{1}{4}\sum_{i, s, t}[-A_{iit}c(\widetilde{e_{t}})
  +A_{isi}c(\widetilde{e_{s}}) -A_{iss}c(\widetilde{e_{i}})+2A_{iii}c(\widetilde{e_{i}})]$, we get
\begin{lem} The following identities hold:
\begin{align}
\sigma_{-2}((D_{T}^{*}D_{T})^{-1})&=|\xi|^{2}=\sigma_{-2}(D^{-2})=|\xi|^{-2};\\
\sigma_{-3}((D_{T}^{*}D_{T})^{-1})&=-\sqrt{-1}|\xi|^{-4}\xi_k(\Gamma^k-2\delta^k)
-\sqrt{-1}|\xi|^{-6}2\xi^j\xi_\alpha\xi_\beta\partial_jg^{\alpha\beta}\nonumber\\
&-\big((u-v)\sqrt{-1}c(\xi)+\sqrt{-1}c(\xi)(u+v)  \big)|\xi|^{-4}.
\end{align}
\end{lem}

 By  Lemma 4.2, Lemma 4.4 and $\sigma(p_{1}\circ p_{2})=\sum_{\alpha}\frac{1}{\alpha!}\partial^{\alpha}_{\xi}[\sigma(p_{1})]
 D_x^{\alpha}[\sigma(p_{2})]$, then
\begin{lem} The following identities hold:
\begin{align}
\sigma_{0}(\widetilde{\nabla}_{X}\widetilde{\nabla}_{Y}(D_{T}^{*}D_{T})^{-1})=&
-\sum_{j,l=1}^nX_jY_l\xi_j\xi_l|\xi|^{-2};\\
\sigma_{-1}(\widetilde{\nabla}_{X}\widetilde{\nabla}_{Y}(D_{T}^{*}D_{T})^{-1})=&
\sigma_{2}(\widetilde{\nabla}_{X}\widetilde{\nabla}_{Y})\sigma_{-3}((D_{T}^{*}D_{T})^{-1})
+\sigma_{1}(\widetilde{\nabla}_{X}\widetilde{\nabla}_{Y})\sigma_{-2}((D_{T}^{*}D_{T})^{-1})\nonumber\\
&+\sum_{j=1}^{n}\partial_{\xi_{j}}\big[\sigma_{2}(\widetilde{\nabla}_{X}\widetilde{\nabla}_{Y})\big]
D_{x_{j}}\big[\sigma_{-2}((D_{T}^{*}D_{T})^{-1})\big].
\end{align}
\end{lem}

Since $\Phi$ is a global form on $\partial M$, so for any fixed point $x_{0}\in\partial M$, we can choose the normal coordinates
$U$ of $x_{0}$ in $\partial M$(not in $M$) and compute $\Phi(x_{0})$ in the coordinates $\widetilde{U}=U\times [0,1)$ and the metric
$\frac{1}{h(x_{n})}g^{\partial M}+\texttt{d}x _{n}^{2}$. The dual metric of $g^{\partial M}$ on $\widetilde{U}$ is
$\frac{1}{\tilde{h}(x_{n})}g^{\partial M}+\texttt{d}x _{n}^{2}.$ Write
$g_{ij}^{M}=g^{M}(\frac{\partial}{\partial x_{i}},\frac{\partial}{\partial x_{j}})$;
$g^{ij}_{M}=g^{M}(d x_{i},dx_{j})$, then

\begin{equation}
[g_{i,j}^{M}]=
\begin{bmatrix}\frac{1}{h( x_{n})}[g_{i,j}^{\partial M}]&0\\0&1\end{bmatrix};\quad
[g^{i,j}_{M}]=\begin{bmatrix} h( x_{n})[g^{i,j}_{\partial M}]&0\\0&1\end{bmatrix},
\end{equation}
and
\begin{equation}
\partial_{x_{s}} g_{ij}^{\partial M}(x_{0})=0,\quad 1\leq i,j\leq n-1;\quad g_{i,j}^{M}(x_{0})=\delta_{ij}.
\end{equation}

Let $\{e_{1},\cdots, e_{n-1}\}$ be an orthonormal frame field in $U$ about $g^{\partial M}$ which is parallel along geodesics and
$e_{i}=\frac{\partial}{\partial x_{i}}(x_{0})$, then $\{\widetilde{e_{1}}=\sqrt{h(x_{n})}e_{1}, \cdots,
\widetilde{e_{n-1}}=\sqrt{h(x_{n})}e_{n-1},\widetilde{e_{n}}=dx_{n}\}$ is the orthonormal frame field in $\widetilde{U}$ about $g^{M}.$
Locally $S(TM)|\widetilde{U}\cong \widetilde{U}\times\wedge^{*}_{C}(\frac{n}{2}).$ Let $\{f_{1},\cdots,f_{n}\}$ be the orthonormal basis of
$\wedge^{*}_{C}(\frac{n}{2})$. Take a spin frame field $\sigma: \widetilde{U}\rightarrow Spin(M)$ such that
$\pi\sigma=\{\widetilde{e_{1}},\cdots, \widetilde{e_{n}}\}$ where $\pi: Spin(M)\rightarrow O(M)$ is a double covering, then
$\{[\sigma, f_{i}], 1\leq i\leq 4\}$ is an orthonormal frame of $S(TM)|_{\widetilde{U}}.$ In the following, since the global form $\Phi$
is independent of the choice of the local frame, so we can compute $\texttt{tr}_{S(TM)}$ in the frame $\{[\sigma, f_{i}], 1\leq i\leq 4\}$.
Let $\{E_{1},\cdots,E_{n}\}$ be the canonical basis of $R^{n}$ and
$c(E_{i})\in cl_{C}(n)\cong Hom(\wedge^{*}_{C}(\frac{n}{2}),\wedge^{*}_{C}(\frac{n}{2}))$ be the Clifford action. By \cite{Y}, then

\begin{equation}
c(\widetilde{e_{i}})=[(\sigma,c(E_{i}))]; \quad c(\widetilde{e_{i}})[(\sigma, f_{i})]=[\sigma,(c(E_{i}))f_{i}]; \quad
\frac{\partial}{\partial x_{i}}=[(\sigma,\frac{\partial}{\partial x_{i}})],
\end{equation}
then we have $\frac{\partial}{\partial x_{i}}c(\widetilde{e_{i}})=0$ in the above frame. By Lemma 2.2 in \cite{Wa3}, we have

\begin{lem}\label{le:32}
With the metric $\frac{1}{h(x_{n})}g^{\partial M}+\texttt{d}x _{n}^{2}$ on $M$ near the boundary
\begin{eqnarray}
&&  \partial x_{j}(|\xi|^{2}_{g^{M}})(x_{0})=0,
     \quad if~~ j<n ;\quad  =h'(0)|\xi'|^{2}_{g^{\partial M}},\quad  if~~ j=n.  \nonumber\\
&&  \partial x_{j}(c(\xi))(x_{0})=0,
     \quad if~~ j<n ;\quad =\partial x_{n}(c(\xi'))(x_{0}),\quad if~~ j=n.
\end{eqnarray}
where $\xi=\xi'+\xi_{n}\texttt{d}x_{n}$
\end{lem}

Now we  need to compute $\int_{\partial M} \Phi$. When $n=4$, then ${\rm tr}_{S(TM)}[{\rm \texttt{id}}]={\rm dim}(\wedge^*(\mathbb{R}^2))=4$, the sum is taken over $
r+l-k-j-|\alpha|=-3,~~r\leq 0,~~l\leq-2,$ then we have the following five cases:

 {\bf case a)~I)}~$r=0,~l=-2,~k=j=0,~|\alpha|=1$.

 By (3.10), we get
\begin{equation}
\label{b24}
\Phi_1=-\int_{|\xi'|=1}\int^{+\infty}_{-\infty}\sum_{|\alpha|=1}
\mathrm{Tr}[\partial^\alpha_{\xi'}\pi^+_{\xi_n}\sigma_{0}(\widetilde{\nabla}_{X}\widetilde{\nabla}_{Y}(D_{T}^{*}D_{T})^{-1})\times
 \partial^\alpha_{x'}\partial_{\xi_n}\sigma_{-2}((D_{T}^{*}D_{T})^{-1})](x_0)d\xi_n\sigma(\xi')dx'.
\end{equation}
By Lemma 2.2 in \cite{Wa3}, for $i<n$, then
\begin{equation}
\label{b25}
\partial_{x_i}\sigma_{-2}((D_{T}^{*}D_{T})^{-1})(x_0)=
\partial_{x_i}(|\xi|^{-2})(x_0)=
-\frac{\partial_{x_i}(|\xi|^{2})(x_0)}{|\xi|^4}=0,
\end{equation}
 so $\Phi_1=0$.

 {\bf case a)~II)}~$r=0,~l=-2,~k=|\alpha|=0,~j=1$.

 By (3.10), we get
\begin{equation}
\label{b26}
\Phi_2=-\frac{1}{2}\int_{|\xi'|=1}\int^{+\infty}_{-\infty}
\mathrm{Tr}[\partial_{x_n}\pi^+_{\xi_n}\sigma_{0}(\widetilde{\nabla}_{X}\widetilde{\nabla}_{Y}(D_{T}^{*}D_{T})^{-1})\times
\partial_{\xi_n}^2\sigma_{-2}((D_{T}^{*}D_{T})^{-1}))](x_0)d\xi_n\sigma(\xi')dx'.
\end{equation}
By Lemma 4.4, we have
\begin{eqnarray}\label{b237}
\partial_{\xi_n}^2\sigma_{-2}((D_{T}^{*}D_{T})^{-1}))(x_0)=\partial_{\xi_n}^2(|\xi|^{-2})(x_0)=\frac{6\xi_n^2-2}{(1+\xi_n^2)^3}.
\end{eqnarray}
It follows that
\begin{eqnarray}\label{b27}
\partial_{x_n}\sigma_{0}(\widetilde{\nabla}_{X}\widetilde{\nabla}_{Y}(D_{T}^{*}D_{T})^{-1})(x_0)
=\partial_{x_n}(-\sum_{j,l=1}^nX_jY_l\xi_j\xi_l|\xi|^{-2})=\frac{\sum_{j,l=1}^nX_jY_l\xi_j\xi_lh'(0)|\xi'|^2}{(1+\xi_n^2)^2}.
\end{eqnarray}
By integrating formula, we obtain
\begin{align}\label{b28}
\pi^+_{\xi_n}\partial_{x_n}\sigma_{0}(\widetilde{\nabla}_{X}\widetilde{\nabla}_{Y}(D_{T}^{*}D_{T})^{-1})(x_0)
&=\partial_{x_n}\pi^+_{\xi_n}\sigma_{0}(\widetilde{\nabla}_{X}\widetilde{\nabla}_{Y}(D_{T}^{*}D_{T})^{-1})\nonumber\\
&=-\frac{i\xi_n}{4(\xi_n-i)^2}\sum_{j,l=1}^{n-1}X_jY_l\xi_j\xi_lh'(0)+\frac{2-i\xi_n}{4(\xi_n-i)^2}X_nY_nh'(0)\nonumber\\
&-\frac{i}{4(\xi_n-i)^2}\sum_{j=1}^{n-1}X_jY_n\xi_j-\frac{i}{4(\xi_n-i)^2}\Sigma_{l=1}^{n-1}X_nY_l\xi_l.
\end{align}
We note that $i<n,~\int_{|\xi'|=1}\xi_{i_{1}}\xi_{i_{2}}\cdots\xi_{i_{2d+1}}\sigma(\xi')=0$,
so we omit some items that have no contribution for computing {\bf case a)~II)}.
From (4.29) and (4.31), we obtain
\begin{align}\label{33}
&\mathrm{Tr} [\partial_{x_n}\pi^+_{\xi_n}\sigma_{0}(\widetilde{\nabla}_{X}\widetilde{\nabla}_{Y}(D_{T}^{*}D_{T})^{-1})\times
\partial_{\xi_n}^2\sigma_{-2}((D_{T}^{*}D_{T})^{-1})](x_0)\nonumber\\
&=2\frac{1+\xi_ni-3\xi_n^3i-i}{(\xi_n-i)^5(\xi_n+i)^3}\Sigma_{j,l=1}^{n-1}X_jY_l\xi_j\xi_lh'(0)
+2\frac{1+\xi_ni-3\xi_n^3i-i}{(\xi_n-i)^5(\xi_n+i)^3}X_nY_nh'(0)\nonumber\\
&+2\frac{(1-3\xi_n^2)i}{(\xi_n-i)^5(\xi_n+i)^3}\Sigma_{j=1}^{n-1}X_jY_n\xi_j
+2\frac{(1-3\xi_n^2)i}{(\xi_n-i)^5(\xi_n+i)^3}\Sigma_{l=1}^{n-1}X_nY_l\xi_l.
\end{align}
Therefore, we get
\begin{align}\label{35}
\Phi_2=&-\frac{1}{2}\int_{|\xi'|=1}\int^{+\infty}_{-\infty}\bigg\{2\frac{1+\xi_ni-3\xi_n^3i-i}{(\xi_n-i)^5(\xi_n+i)^3}
\sum_{j,l=1}^{n-1}X_jY_l\xi_j\xi_lh'(0)+2\frac{1+\xi_ni-3\xi_n^3i-i}{(\xi_n-i)^5(\xi_n+i)^3}X_nY_nh'(0)\bigg\}d\xi_n\sigma(\xi')dx'\nonumber\\
=&-\sum_{j,l=1}^{n-1}X_jY_lh'(0)\int_{|\xi'|=1}\int_{\Gamma^{+}}\frac{1+\xi_ni-3\xi_n^3i-i}{(\xi_n-i)^5(\xi_n+i)^3}\xi_j\xi_ld\xi_{n}\sigma(\xi')dx'\nonumber\\
&-X_nY_nh'(0)\Omega_3\int_{\Gamma^{+}}\frac{1+\xi_ni-3\xi_n^3i-i}{(\xi_n-i)^5(\xi_n+i)^3}d\xi_{n}dx'\nonumber\\
=&-\sum_{j,l=1}^{n-1}X_jY_lh'(0)\frac{4\pi}{3}\frac{2\pi i}{4!}\left[\frac{1+\xi_ni-3\xi_n^3i-i}{(\xi_n+i)^3}\right]^{(4)}\bigg|_{\xi_n=i}dx'\nonumber\\
&-X_nY_nh'(0)\Omega_3\frac{2\pi i}{4!}\left[\frac{1+\xi_ni-3\xi_n^3i-i}{(\xi_n+i)^3}\right]^{(4)}\bigg|_{\xi_n=i}dx'\nonumber\\
=& \frac{13\pi^{2}}{24}\sum_{j=1}^{n-1}X_jY_jh'(0)dx'+\frac{13}{32}X_nY_n h'(0)\pi\Omega_3dx',
\end{align}
where ${\rm \Omega_{3}}$ is the canonical volume of $S^{2}.$

  {\bf case a)~III)}~$r=0,~l=-2,~j=|\alpha|=0,~k=1$.

By (3.10), we get
\begin{align}\label{36}
\Phi_3&=-\frac{1}{2}\int_{|\xi'|=1}\int^{+\infty}_{-\infty}
\mathrm{Tr} [\partial_{\xi_n}\pi^+_{\xi_n}\sigma_{0}(\widetilde{\nabla}_{X}\widetilde{\nabla}_{Y}(D_{T}^{*}D_{T})^{-1})\times
\partial_{\xi_n}\partial_{x_n}\sigma_{-2}(D_{T}^{*}D_{T})^{-1})](x_0)d\xi_n\sigma(\xi')dx'\nonumber\\
&=\frac{1}{2}\int_{|\xi'|=1}\int^{+\infty}_{-\infty}
\mathrm{Tr}[\partial_{\xi_n}^2\pi^+_{\xi_n}\sigma_{0}(\widetilde{\nabla}_{X}\widetilde{\nabla}_{Y}(D_{T}^{*}D_{T})^{-1}
)\times
\partial_{x_n}\sigma_{-2}(D_{T}^{*}D_{T})^{-1})](x_0)d\xi_n\sigma(\xi')dx'.
\end{align}
 By Lemma 4.4, we have
\begin{eqnarray}\label{37}
\partial_{x_n}\sigma_{-2}(D_{T}^{*}D_{T})^{-1})(x_0)|_{|\xi'|=1}
=-\frac{h'(0)}{(1+\xi_n^2)^2}.
\end{eqnarray}
An easy calculation gives
\begin{align}\label{38}
\pi^+_{\xi_n}\sigma_{0}(\widetilde{\nabla}_{X}\widetilde{\nabla}_{Y}(D_{T}^{*}D_{T})^{-1})(x_0)|_{|\xi'|=1}&
=\frac{i}{2(\xi_n-i)}\sum_{j,l=1}^{n-1}X_jY_l\xi_j\xi_l-\frac{1}{2(\xi_n-i)}X_nY_n\nonumber\\
&-\frac{1}{2(\xi_n-i)}\sum_{j=1}^{n-1}X_jY_n\xi_j-\frac{1}{2(\xi_n-i)}\sum_{l=1}^{n-1}X_nY_l\xi_l.
\end{align}
Also, straightforward computations yield
\begin{align}\label{mmmmm}
\partial_{\xi_n}^2\pi^+_{\xi_n}\sigma_{0}(\widetilde{\nabla}_{X}\widetilde{\nabla}_{Y}(D_{T}^{*}D_{T})^{-1})(x_0)|_{|\xi'|=1}
=\frac{i}{(\xi_n-i)^3}\sum_{j,l=1}^{n-1}X_jY_l\xi_j\xi_l-\frac{1}{(\xi_n-i)^3}X_nY_n.
\end{align}
From (4.35) and (4.37), we obtain
\begin{align}\label{39}
&\mathrm{Tr}[\partial_{\xi_n}\pi^+_{\xi_n}\sigma_{0}(\widetilde{\nabla}_{X}\widetilde{\nabla}_{Y}(D_{T}^{*}D_{T})^{-1})\times
\partial_{\xi_n}\partial_{x_n}\sigma_{-2}(D_{T}^{*}D_{T})^{-1})](x_0)\nonumber\\
&=-4\frac{h'(0)i}{(\xi_n-i)^5(\xi_n+i)^2}\sum_{j,l=1}^{n-1}X_jY_l\xi_j\xi_l+4\frac{h'(0)}{(\xi_n-i)^5(\xi_n+i)^2}X_nY_n.
\end{align}
Therefore, we get
\begin{align}\label{41}
\Phi_3&=\frac{1}{2}\int_{|\xi'|=1}\int^{+\infty}_{-\infty}
\bigg(-4\frac{h'(0)i}{(\xi_n-i)^5(\xi_n+i)^2}
\sum_{j,l=1}^{n-1}X_jY_l\xi_j\xi_l+4\frac{h'(0)}{(\xi_n-i)^5(\xi_n+i)^2}X_nY_n\bigg)d\xi_n\sigma(\xi')dx'\nonumber\\
&=-2\sum_{j,l=1}^{n-1}X_jY_lh'(0)\Omega_3\int_{\Gamma^{+}}\frac{i}{(\xi_n-i)^5(\xi_n+i)^2}\xi_j\xi_ld\xi_{n}dx'+2X_nY_nh'(0)\Omega_3\int_{\Gamma^{+}}\frac{1}{(\xi_n-i)^5(\xi_n+i)^2}d\xi_{n}dx'\nonumber\\
&=-2\sum_{j,l=1}^{n-1}X_jY_lh'(0)\Omega_3\frac{2\pi i}{4!}\left[\frac{i}{(\xi_n+i)^2}\right]^{(4)}
\bigg|_{\xi_n=i}dx'+2X_nY_nh'(0)\Omega_3\frac{2\pi i}{4!}\left[\frac{1}{(\xi_n+i)^2}\right]^{(4)}\bigg|_{\xi_n=i}dx'\nonumber\\
&= \frac{5\pi^{2}}{12}\sum_{j=1}^{n-1}X_jY_jh'(0)dx'+\frac{5i}{16}X_nY_n h'(0)\pi\Omega_3dx'.
\end{align}

 {\bf case b)}~$r=0,~l=-3,~k=j=|\alpha|=0$.

 By (3.10), we get
\begin{align}\label{42}
\Phi_4&=-i\int_{|\xi'|=1}\int^{+\infty}_{-\infty}\mathrm{Tr} [\pi^+_{\xi_n}
\sigma_{0}(\widetilde{\nabla}_{X}\widetilde{\nabla}_{Y}(D_{T}^{*}D_{T})^{-1})\times
\partial_{\xi_n}\sigma_{-3}(D_{T}^{*}D_{T})^{-1})](x_0)d\xi_n\sigma(\xi')dx'\nonumber\\
&=i\int_{|\xi'|=1}\int^{+\infty}_{-\infty}\mathrm{Tr} [\partial_{\xi_n}\pi^+_{\xi_n}
\sigma_{0}(\widetilde{\nabla}_{X}\widetilde{\nabla}_{Y}(D_{T}^{*}D_{T})^{-1})\times
\sigma_{-3}(D_{T}^{*}D_{T})^{-1})](x_0)d\xi_n\sigma(\xi')dx'.
\end{align}
 By Lemma 4.4, we have
\begin{align}\label{43}
\sigma_{-3}(D_{T}^{*}D_{T})^{-1})(x_0)|_{|\xi'|=1}
=&-\frac{i}{(1+\xi_n^2)^2}\left(-\frac{1}{2}h'(0)\sum_{k<n}\xi_nc(e_k)c(e_n)+\frac{5}{2}h'(0)\xi_n\right)
-\frac{2ih'(0)\xi_n}{(1+\xi_n^2)^3}\nonumber\\
&-\big((u-v)\sqrt{-1}c(\xi)+\sqrt{-1}c(\xi)(u+v)\big)|\xi|^{-4}.
\end{align}
\begin{align}\label{45}
\partial_{\xi_n}\pi^+_{\xi_n}\sigma_{0}(\widetilde{\nabla}_{X}\widetilde{\nabla}_{Y}(D_{T}^{*}D_{T})^{-1})(x_0)|_{|\xi'|=1}
&=-\frac{i}{2(\xi_n-i)^2}\sum_{j,l=1}^{n-1}X_jY_l\xi_j\xi_l-\frac{1}{2(\xi_n-i)^2}X_nY_n\nonumber\\
&+\frac{1}{2(\xi_n-i)^2}\sum_{j=1}^{n-1}X_jY_n\xi_j+\frac{1}{2(\xi_n-i)^2}\sum_{l=1}^{n-1}X_nY_l\xi_l.
\end{align}
We note that $i<n,~\int_{|\xi'|=1}\xi_{i_{1}}\xi_{i_{2}}\cdots\xi_{i_{2d+1}}\sigma(\xi')=0$,
and $\sum_{i\neq s\neq t}Tr[A_{ist}c(\widetilde{e_{i}})c(\widetilde{e_{s}})c(\widetilde{e_{t}})c(\xi')]=0$,
so we omit some items that have no contribution for computing {\bf case b)}.
Then, we have
\begin{align}\label{39}
&\mathrm{Tr}[\partial_{\xi_n}\pi^+_{\xi_n}\sigma_{0}(\widetilde{\nabla}_{X}\widetilde{\nabla}_{Y}(D_{T}^{*}D_{T})^{-1})\times
\sigma_{-3}(D_{T}^{*}D_{T})^{-1})](x_0)\nonumber\\
&=-\frac{h'(0)(5\xi_n^2-5+4\xi_n)} {(\xi_n-i)^5(\xi_n+i)^3}\sum_{j,l=1}^{n-1}X_jY_l\xi_j\xi_l
+\frac{h'(0)i(5\xi_n^3-\xi_n)}{(\xi_n-i)^5(\xi_n+i)^3}X_nY_n.
\end{align}
Therefore, we get
\begin{align}\label{41}
\Phi_4&=i\int_{|\xi'|=1}\int^{+\infty}_{-\infty}
\bigg(-\frac{h'(0)(5\xi_n^2-5+4\xi_n)} {(\xi_n-i)^5(\xi_n+i)^3}\Sigma_{j,l=1}^{n-1}X_jY_l\xi_j\xi_l+\frac{h'(0)i(5\xi_n^3-\xi_n)}{(\xi_n-i)^5(\xi_n+i)^3}X_nY_n\bigg)d\xi_n\sigma(\xi')dx'\nonumber\\
&=-i\Sigma_{j,l=1}^{n-1}X_jY_lh'(0)\Omega_3\int_{\Gamma^{+}}\frac{5\xi_n^2-5+4\xi_n}{(\xi_n-i)^5(\xi_n+i)^2}\xi_j\xi_ld\xi_{n}dx'+iX_nY_nh'(0)\Omega_3\int_{\Gamma^{+}}\frac{5\xi_n^3-\xi_n}{(\xi_n-i)^5(\xi_n+i)^2}d\xi_{n}dx'\nonumber\\
&=-i\Sigma_{j,l=1}^{n-1}X_jY_lh'(0)\Omega_3\frac{2\pi i}{4!}\left[\frac{5\xi_n^2-5+4\xi_n}{(\xi_n+i)^2}\right]^{(4)}\bigg|_{\xi_n=i}dx'+iX_nY_nh'(0)\Omega_3\frac{2\pi i}{4!}\left[\frac{5\xi_n^3-\xi_n}{(\xi_n+i)^2}\right]^{(4)}\bigg|_{\xi_n=i}dx'\nonumber\\
&= \frac{(1-5i)\pi^{2}}{12}\sum_{j=1}^{n-1}X_jY_jh'(0)dx'+\frac{11i}{16}X_nY_n h'(0)\pi\Omega_3dx'.
\end{align}

 {\bf  case c)}~$r=-1,~\ell=-2,~k=j=|\alpha|=0$.

By (3.10), we get
\begin{align}\label{61}
\Phi_5=-i\int_{|\xi'|=1}\int^{+\infty}_{-\infty}\mathrm{Tr} [\pi^+_{\xi_n}
\sigma_{-1}(\widetilde{\nabla}_{X}\widetilde{\nabla}_{Y}(D_{T}^{*}D_{T})^{-1})\times
\partial_{\xi_n}\sigma_{-2}(D_{T}^{*}D_{T})^{-1})](x_0)d\xi_n\sigma(\xi')dx'.
\end{align}
By Lemma 4.4, we have
\begin{align}\label{62}
\partial_{\xi_n}\sigma_{-2}(D_{T}^{*}D_{T})^{-1})(x_0)|_{|\xi'|=1}=-\frac{2\xi_n}{(\xi_n^2+1)^2}.
\end{align}
Since
\begin{align}
\sigma_{-1}(\widetilde{\nabla}_{X}\widetilde{\nabla}_{Y}(D_{T}^{*}D_{T})^{-1})(x_0)|_{|\xi'|=1}
=&
\sigma_{2}(\widetilde{\nabla}_{X}\widetilde{\nabla}_{Y})\sigma_{-3}((D_{T}^{*}D_{T})^{-1})
+\sigma_{1}(\widetilde{\nabla}_{X}\widetilde{\nabla}_{Y})\sigma_{-2}((D_{T}^{*}D_{T})^{-1})\nonumber\\
&+\sum_{j=1}^{n}\partial_{\xi_{j}}\big[\sigma_{2}(\widetilde{\nabla}_{X}\widetilde{\nabla}_{Y})\big]
D_{x_{j}}\big[\sigma_{-2}((D_{T}^{*}D_{T})^{-1})\big].
\end{align}
 Explicit representation the first item of (4.47),
\begin{align}
&\sigma_{2}(\widetilde{\nabla}_{X}\widetilde{\nabla}_{Y})\sigma_{-3}((D_{T}^{*}D_{T})^{-1})(x_0)|_{|\xi'|=1}\nonumber\\
=&-\sum_{j,l=1}^{n}X_jY_l\xi_j\xi_l\times\Big(-\sqrt{-1}|\xi|^{-4}\xi_k(\Gamma^k-2\delta^k)
-\sqrt{-1}|\xi|^{-6}2\xi^j\xi_\alpha\xi_\beta\partial_jg^{\alpha\beta}\nonumber\\
&-\big((u-v)\sqrt{-1}c(\xi)+\sqrt{-1}c(\xi)(u+v)  \big)|\xi|^{-4}\Big)\nonumber\\
=&-\sum_{j,l=1}^{n}X_jY_l\xi_j\xi_l\times\Big( -\frac{i}{(1+\xi_n^2)^2}
\big(-\frac{1}{2}h'(0)\sum_{k<n}\xi_nc(e_k)c(e_n)+\frac{5}{2}h'(0)\xi_n\big)
-\frac{2ih'(0)\xi_n}{(1+\xi_n^2)^3}\nonumber\\
&-\big((u-v)\sqrt{-1}c(\xi)+\sqrt{-1}c(\xi)(u+v)\big)|\xi|^{-4}\Big).
\end{align}
Explicit representation the second item of (4.47),
\begin{align}
&\sigma_{1}(\widetilde{\nabla}_{X}\widetilde{\nabla}_{Y})\sigma_{-2}((D_{T}^{*}D_{T})^{-1})(x_0)|_{|\xi'|=1}\nonumber\\
=&\Big(\sqrt{-1}\sum_{j,l=1}^nX_j\frac{\partial_{Y_l}}{\partial_{x_j}}\sqrt{-1}\xi_l
+\sqrt{-1}\sum_jA(Y)X_j\xi_j+\sqrt{-1}\sum_lA(Y)Y_l\xi_l\nonumber\\
&+\sum_j  \overline{T}(X,\cdot)Y_j\sqrt{-1}\xi_j+\sum_j \overline{T}(Y,\cdot)X_j\sqrt{-1} \xi_j\Big)\times|\xi|^{-2}.
\end{align}
Explicit representation the third item of (4.47),
\begin{align}
&\sum_{j=1}^{n}\sum_{\alpha}\frac{1}{\alpha!}\partial^{\alpha}_{\xi}\big[\sigma_{2}(\widetilde{\nabla}_{X}\widetilde{\nabla}_{Y})\big]
D_x^{\alpha}\big[\sigma_{-2}((D_{T}^{*}D_{T})^{-1})\big](x_0)|_{|\xi'|=1}\nonumber\\
=&\sum_{j=1}^{n}\partial_{\xi_{j}}\big[\sigma_{2}(\widetilde{\nabla}_{X}\widetilde{\nabla}_{Y})\big]
(-\sqrt{-1})\partial_{x_{j}}\big[\sigma_{-2}((D_{T}^{*}D_{T})^{-1})\big]\nonumber\\
=&\sum_{j=1}^{n}\partial_{\xi_{j}}\big[-\sum_{j,l=1}^nX_jY_l\xi_j\xi_l\big]
(-\sqrt{-1})\partial_{x_{j}}\big[|\xi|^{-2}\big]\nonumber\\
=&\sum_{j=1}^{n}\sum_{l=1}^{n}\sqrt{-1}(x_{j}Y_l+x_{l}Y_j)\xi_{l}\partial_{x_{j}}(|\xi|^{-2}).
\end{align}
We note that $i<n,~\int_{|\xi'|=1}\xi_{i_{1}}\xi_{i_{2}}\cdots\xi_{i_{2d+1}}\sigma(\xi')=0$,
and $\sum_{i\neq s\neq t}Tr[A_{ist}c(\widetilde{e_{i}})c(\widetilde{e_{s}})c(\widetilde{e_{t}})c(\xi')]=0$,
so we omit some items that have no contribution for computing {\bf case c)}.
An easy calculation gives
\begin{align}\label{71}
&\mathrm{Tr}[\pi^+_{\xi_n}\sigma_{-1}(\sigma_{2}(\widetilde{\nabla}_{X}\widetilde{\nabla}_{Y})\sigma_{-3}((D_{T}^{*}D_{T})^{-1}))\times
\partial_{\xi_n}\sigma_{-2}(D^{-2})](x_0)|_{|\xi'|=1}\nonumber\\
=&\frac{h'(0)(2\xi_n^2-\xi_n-2\xi_ni)}{(\xi_n-i)^4(\xi_n+i)^2}\sum_{j,l=1}^{n-1}X_jY_l\xi_j\xi_l\sum_{k<n}\xi_kc(e_k)c(e_n)
+\frac{h'(0)(17\xi_ni-\xi_n^2+4\xi_n^3i)}{(\xi_n-i)^5(\xi_n+i)^2}\sum_{j,l=1}^{n-1}X_jY_l\xi_j\xi_l,
\end{align}
and
\begin{align}\label{71}
\mathrm{Tr}(\overline{T}(X,\cdot,\cdot))=\mathrm{Tr}(\frac{3}{2}\sum_{ 1\leq i<j\leq n}T(X,e_{i},e_{j})c(e_{i})c (e_{j})-\frac{1}{2}c(V)c(X)
  -\frac{1}{2}\langle V, X \rangle)=0.
\end{align}
  Also, straightforward computations yield
\begin{align}\label{71}
&\mathrm{Tr}[\pi^+_{\xi_n}\sigma_{-1}(\sigma_{1}(\widetilde{\nabla}_{X}\widetilde{\nabla}_{Y})\sigma_{-2}((D_{T}^{*}D_{T})^{-1}))\times
\partial_{\xi_n}\sigma_{-2}(D^{-2})](x_0)|_{|\xi'|=1}\nonumber\\
=&X_n\frac{\partial Y_n}{\partial x_n} \frac{-2\xi_n}{(\xi_n^2+1)^2},
\end{align}
and
\begin{align}\label{71}
&\mathrm{Tr}[\pi^+_{\xi_n}\sigma_{-1}(\sum_{j=1}^{n}\sum_{\alpha}\frac{1}{\alpha!}\partial^{\alpha}_{\xi}
\big[\sigma_{2}(\widetilde{\nabla}_{X}\widetilde{\nabla}_{Y})\big]
D_x^{\alpha}\big[\sigma_{-2}((D_{T}^{*}D_{T})^{-1})\big])\times
\partial_{\xi_n}\sigma_{-2}(D^{-2})](x_0)|_{|\xi'|=1}\nonumber\\
=&2iX_nY_nh'(0)\xi_n\times \frac{-2\xi_n}{(\xi_n^2+1)^2}.
\end{align}
Substituting (4.51),(4.53) and (4.54) into (4.45) yields
\begin{align}\label{74}
\Phi_5=&\Big(\frac{5i-13}{6}\sum_{j=1}^{n-1}X_jY_j +\frac{3-96i}{8}X_nY_n \Big) h'(0)\pi^2dx'\nonumber\\
&-X_n\frac{\partial Y_n}{\partial x_n}\frac{\pi}{2}\Omega_3dx'
\end{align}
Let $X=X^T+X_n\partial_n,~Y=Y^T+Y_n\partial_n,$ then we have $\sum_{j=1}^{n-1}X_jY_j=g(X^T,Y^T).$
Now $\Phi$ is the sum of the cases (a), (b) and (c). Therefore, we get
\begin{align}\label{795}
\Phi=&\sum_{i=1}^5\Phi_i= \frac{15-362i}{32}X_nY_nh'(0)\pi\Omega_3dx'+\frac{(10i-27)\pi^{2}}{24}g(X^T,Y^T) h'(0)  dx'
-X_n\frac{\partial Y_n}{\partial x_n}\frac{\pi}{2}\Omega_3dx'
\end{align}
Then we obtain following theorem
\begin{thm}\label{thmb1}
 Let $M$ be a 4-dimensional compact manifold without boundary and $\widetilde{\nabla}$ be an orthogonal
connection with torsion. Then we get the volumes  associated to $\widetilde{\nabla}_{X}\widetilde{\nabla}_{Y}(D_{F}^{*}D_{F})^{-1}$
and $D_{T}^{*}D_{T}$ on compact manifolds with boundary
\begin{align}
\label{b2773}
&\widetilde{{\rm Wres}}[\pi^+(\widetilde{\nabla}_{X}\widetilde{\nabla}_{Y}(D_{F}^{*}D_{F})^{-1})\circ\pi^+(D_{T}^{*}D_{T})]\nonumber\\
=&\frac{4\pi^{2}}{3}\int_{M}\Big(Ric(X,Y)-\frac{1}{2}sg(X,Y)\Big) vol_{g}\nonumber\\
&+\int_{M}
\Big( -\frac{1}{2}R^{g}-3\mathrm{div}^{g}(X)+3\parallel T\parallel^{2}+9\parallel X\parallel^{2}\Big)g(X,Y) vol_{g}\nonumber\\
&+\int_{\partial M}
\Big[ \Big(\frac{15-362i}{32}X_nY_n\pi
h'(0)
-X_n\frac{\partial Y_n}{\partial x_n}\frac{\pi}{2}\Big) \Omega_3  +\frac{(10i-27)\pi^{2}}{24}g(X^T,Y^T) h'(0) \Big]
vol_{\partial M},
\end{align}
where $R_{g}$ denotes the curvature tensor and $s$ is the scalar curvature.
\end{thm}

 \section{Residue for Dirac operators
 with torsion  $\widetilde{\nabla}_{X}\widetilde{\nabla}_{Y}D_{T}^{-1}$ and $(D_{T}^{*}D_{T}D_{T}^{*})^{-1}$ }

In this section, we compute the 4-dimension volume for  Dirac operators
 with torsion  $\widetilde{\nabla}_{X}\widetilde{\nabla}_{Y}D_{T}^{-1}$ and $(D_{T}^{*}D_{T}D_{T}^{*})^{-1}$ .
 Since $[\sigma_{-4}(\widetilde{\nabla}_{X}\widetilde{\nabla}_{Y}D_{T}^{-1}
  \circ (D_{T}^{*}D_{T}D_{T}^{*})^{-1})]|_{M}$ has the same expression as $[\sigma_{-4}(\widetilde{\nabla}_{X}\widetilde{\nabla}_{Y}D_{T}^{-1}
  \circ (D_{T}^{*}D_{T}D_{T}^{*})^{-1})]|_{M}$ in the case of manifolds without boundary,
so locally we can use Theorem 2.4 to compute the first term.
\begin{thm}
 Let M be a four dimensional compact manifold without boundary and $\widetilde{\nabla}$ be an orthogonal
connection with torsion. Then we get the volumes  associated to $\widetilde{\nabla}_{X}\widetilde{\nabla}_{Y}(D_{F}^{*}D_{F})^{-1}$
and $D_{T}^{*}D_{T}$ on compact manifolds without boundary
 \begin{align}
&Wres[\sigma_{-4}(\widetilde{\nabla}_{X}\widetilde{\nabla}_{Y}D_{T}^{-1}
  \circ (D_{T}^{*}D_{T}D_{T}^{*})^{-1})]\nonumber\\
=&\frac{4\pi^{2}}{3}\int_{M}\Big(Ric(X,Y)-\frac{1}{2}sg(X,Y)\Big) vol_{g}\nonumber\\
&+\int_{M}
\Big( -\frac{1}{2}R^{g}-3div^{g}(X)+3\parallel T\parallel^{2}+9\parallel X\parallel^{2}\Big)g(X,Y) vol_{g},
\end{align}
where $R_{g}$ denotes the curvature tensor and $s$ is the scalar curvature.
\end{thm}
From lemma 4.2 and lemma 4.3, we have
\begin{lem} The following identities hold:
\begin{align}
\sigma_{1}(\widetilde{\nabla}_{X}\widetilde{\nabla}_{Y} D_{T} ^{-1})=&
-\sqrt{-1}\sum_{j,l=1}^nX_jY_l\xi_j\xi_lc(\xi)|\xi|^{-2};\\
\sigma_{0}(\widetilde{\nabla}_{X}\widetilde{\nabla}_{Y}D_{T}^{-1})=&
\sigma_{2}(\widetilde{\nabla}_{X}\widetilde{\nabla}_{Y})\sigma_{-2}(D_{T}^{-1})
+\sigma_{1}(\widetilde{\nabla}_{X}\widetilde{\nabla}_{Y})\sigma_{-1}(D_{T}^{-1})\nonumber\\
&+\sum_{j=1}^{n}\partial _{\xi_{j}}\big[\sigma_{2}(\widetilde{\nabla}_{X}\widetilde{\nabla}_{Y})\big]
D_{x_{j}}\big[\sigma_{-1}(D_{T}^{-1})\big].
\end{align}
\end{lem}

Write
 \begin{eqnarray}
D_x^{\alpha}&=(-i)^{|\alpha|}\partial_x^{\alpha};
~\sigma(D_t^3)=p_3+p_2+p_1+p_0;
~(\sigma(D_t)^{-3})=\sum^{\infty}_{j=3}q_{-j}.
\end{eqnarray}
By the composition formula of pseudodifferential operators, we have
\begin{align}
1=\sigma(D^3\circ D^{-3})
&=\sum_{\alpha}\frac{1}{\alpha!}\partial^{\alpha}_{\xi}[\sigma(D]
D_x^{\alpha}[\sigma(D^{-3})]\nonumber\\
&=(p_3+p_2+p_1+p_0)(q_{-3}+q_{-4}+q_{-5}+\cdots)\nonumber\\
&~~~+\sum_j(\partial_{\xi_j}p_3+\partial_{\xi_j}p_2++\partial_{\xi_j}p_1+\partial_{\xi_j}p_0)
(D_{x_j}q_{-3}+D_{x_j}q_{-4}+D_{x_j}q_{-5}+\cdots)\nonumber\\
&=p_3q_{-3}+(p_3q_{-4}+p_2q_{-3}+\sum_j\partial_{\xi_j}p_3D_{x_j}q_{-3})+\cdots,
\end{align}
so
\begin{equation}
q_{-3}=p_3^{-1};~q_{-4}=-p_3^{-1}[p_2p_3^{-1}+\sum_j\partial_{\xi_j}p_3D_{x_j}(p_-3^{-1})].
\end{equation}
Then it is easy to check that
\begin{lem} The following identities hold:
\begin{align}
&\sigma_{-2}((D_{T}^{*}D_{T})^{-1})=|\xi|^{2};\\
&\sigma_{-3}((D_{T}^{*}D_{T})^{-1})=-\sqrt{-1}|\xi|^{-4}\xi_k(\Gamma^k-2\delta^k)
-\sqrt{-1}|\xi|^{-6}2\xi^j\xi_\alpha\xi_\beta\partial_jg^{\alpha\beta}\nonumber\\
&~~~~~~~~~~~~~~~~~~~~~-\big((u-v)\sqrt{-1}c(\xi)+\sqrt{-1}c(\xi)(u+v)  \big)|\xi|^{-4}\\
&\sigma_{-3}((D_{T}^{*}D_{T}D_{T}^{*})^{-1})=\sqrt{-1}c(\xi)|\xi|^{-4};\\
&\sigma_{-4}((D_{T}^{*}D_{T}D_{T}^{*})^{-1})=\frac{c(\xi)\sigma_2((D_{T}^{*}D_{T}D_{T}^{*})^{-1})c(\xi)}{|\xi|^8}
+\frac{\sqrt{-1}c(\xi)}{|\xi|^8}\bigg(|\xi|^4c(dx_n)\partial_{x_n}c(\xi')
\nonumber\\
 &~~~~~~~~~~~~~~~~~~~~~~~~~~-2h'(0)c(\mathrm{d}x_n)c(\xi)+2\xi_nc(\xi)\partial{x_n}c(\xi')+4\xi_nh'(0)\bigg),
\end{align}
where
 \begin{align}
\sigma_2((D_{T}^{*}D_{T}D_{T}^{*})^{-1})=&c(\xi)(4\sigma^k-2\Gamma^k)\xi_{k}
     -\frac{1}{4}|\xi|^2\sum_{s,t}\omega_{s,t}(\widetilde{e_l})c(e_{l})c(\widetilde{e_s})c(\widetilde{e_t})+(3u-v)|\xi|^2.
\end{align}
\end{lem}
Now we  need to compute $\int_{\partial M} \widetilde{\Phi}$. When $n=4$, then ${\rm tr}_{S(TM)}[{\rm \texttt{id}}]={\rm dim}(\wedge^*(\mathbb{R}^2))=4$, the sum is taken over $
r+l-k-j-|\alpha|=-3,~~r\leq 0,~~l\leq-2,$ then we have the following five cases:

 {\bf case a)~I)}~$r=1,~l=-2,~k=j=0,~|\alpha|=1$.

By (3.12), we get
\begin{equation}
\label{b24}
\widetilde{\Phi}_1=-\int_{|\xi'|=1}\int^{+\infty}_{-\infty}\sum_{|\alpha|=1}
\mathrm{Tr}[\partial^\alpha_{\xi'}\pi^+_{\xi_n}\sigma_{1}(\widetilde{\nabla}_{X}\widetilde{\nabla}_{Y}D_{T}^{-1})\times
 \partial^\alpha_{x'}\partial_{\xi_n}\sigma_{-3}((D_{T}^{*}D_{T}D_{T}^{*})^{-1})](x_0)d\xi_n\sigma(\xi')dx'.
\end{equation}
By Lemma 2.2 in \cite{Wa3}, for $i<n$, then
\begin{equation}
\label{b25}
\partial_{x_i}\sigma_{-3}((D_{T}^{*}D_{T}D_{T}^{*})^{-1})(x_0)=
\partial_{x_i}(\sqrt{-1}c(\xi)|\xi|^{-4})(x_0)=
\sqrt{-1}\frac{\partial_{x_i}c(\xi)}{|\xi|^4}(x_0)
+\sqrt{-1}\frac{c(\xi)\partial_{x_i}(|\xi|^4)}{|\xi|^8}(x_0)
=0,
\end{equation}
\noindent so $\widetilde{\Phi}_1=0$.\\

  {\bf case a)~II)}~$r=1,~l=-3,~k=|\alpha|=0,~j=1$.

  By (3.12), we get
\begin{equation}
\label{b26}
\widetilde{\Phi}_2=-\frac{1}{2}\int_{|\xi'|=1}\int^{+\infty}_{-\infty}
\mathrm{Tr} [\partial_{x_n}\pi^+_{\xi_n}\sigma_{1}(\widetilde{\nabla}_{X}\widetilde{\nabla}_{Y}D_{T}^{-1})\times
\partial_{\xi_n}^2\sigma_{-3}((D_{T}^{*}D_{T}D_{T}^{*})^{-1})](x_0)d\xi_n\sigma(\xi')dx'.
\end{equation}
By Lemma 5.3, we have
\begin{eqnarray}\label{b237}
\partial_{\xi_n}^2\sigma_{-3}((D_{T}^{*}D_{T}D_{T}^{*})^{-1})(x_0)=\partial_{\xi_n}^2(c(\xi)|\xi|^{-4})(x_0)
=\sqrt{-1}\frac{(20\xi_n^2-4)c(\xi')+12(\xi^3-\xi)c(\mathrm{d}x_n)}{(1+\xi_n^2)^4},
\end{eqnarray}
and
\begin{align}\label{b27}
\partial_{x_n}\sigma_{1}(\widetilde{\nabla}_{X}\widetilde{\nabla}_{Y}D_{T}^{-1})(x_0)
&=\partial_{x_n}(-\sqrt{-1}\Sigma_{j,l=1}^nX_jY_l\xi_j\xi_lc(\xi)|\xi|^{-2})\nonumber\\
&=\Sigma_{j,l=1}^nX_jY_l\xi_j\xi_l \left[ \frac{\partial_{x_n}c(\xi')}{1+\xi_n^2}+\frac{c(\xi)h'(0)|\xi'|^2}{(1+\xi_n^2)^2}\right].
\end{align}
Then, we have
\begin{align}\label{b28}
\pi^+_{\xi_n}\partial_{x_n}\sigma_{1}(\widetilde{\nabla}_{X}\widetilde{\nabla}_{Y}D_{T}^{-1})(x_0)
=&\partial_{x_n}\pi^+_{\xi_n}\sigma_{1}(\widetilde{\nabla}_{X}\widetilde{\nabla}_{Y}D_{T}^{-1})(x_0)\nonumber\\
=&\sqrt{-1}\sum_{j,l=1}^{n-1}X_jY_l\xi_j\xi_lh'(0)|\xi'|^2\left[\frac{ic(\xi')}{4(\xi_n-i)}+\frac{c(\xi')+ic(\mathrm{d}x_n)}{4(\xi_n-i)^2}\right]\nonumber\\
&-\sum_{j,l=1}^{n-1}X_jY_l\xi_j\xi_l\frac{\partial_{x_n}c(\xi')}{2(\xi_n-i)}\nonumber\\
&-\sqrt{-1}X_nY_n\left\{\frac{\partial_{x_n}c(\xi')}{2(\xi_n-i)}
 +h'(0)|\xi'|\left[-\frac{2ic(\xi')-3c(\mathrm{d}x_n)}{4(\xi_n-i)}\right.\right.\nonumber\\
&+\left.\left.\frac{[c(\xi')+ic(\mathrm{d}x_n)][i(\xi_n-i)+1]}{4(\xi_n-i)^2}\right]\right\}\nonumber\\
&-\Sigma_{j=1}^{n-1}X_jY_n\xi_j\left[\sqrt{-1}\frac{\partial_{x_n}c(\xi')}{2(\xi_n-i)}
-\frac{\sqrt{-1}h'(0)|\xi'|[c(\xi')+2ic(\mathrm{d}x_n)]}{4(\xi_n-i)}\right.\nonumber\\
&\left.-\frac{[ic(\xi')-c(\mathrm{d}x_n)][i(\xi_n-i)+1]}{(\xi_n-i)^2}\right]\nonumber\\
&-\Sigma_{l=1}^{n-1}X_nY_l\xi_l \left[\sqrt{-1}\frac{\partial_{x_n}c(\xi')}{2(\xi_n-i)}
-\frac{\sqrt{-1}h'(0)|\xi'|[c(\xi')+2ic(\mathrm{d}x_n)]}{4(\xi_n-i)}\right.\nonumber\\
&\left.-\frac{[ic(\xi')-c(\mathrm{d}x_n)][i(\xi_n-i)+1]}{(\xi_n-i)^2}\right] .
\end{align}

We note that $i<n,~\int_{|\xi'|=1}\xi_{i_{1}}\xi_{i_{2}}\cdots\xi_{i_{2d+1}}\sigma(\xi')=0$,
so we omit some items that have no contribution for computing {\bf case a)~II)}.
Then there is the following formula
\begin{align}\label{33}
&\mathrm{Tr}[\partial_{x_n}\pi^+_{\xi_n}\sigma_{1}(\widetilde{\nabla}_{X}\widetilde{\nabla}_{Y}D_{T}^{-1})\times
\partial_{\xi_n}^2\sigma_{-3}((D_{T}^{*}D_{T}D_{T}^{*})^{-1})](x_0)\nonumber\\
&=\Sigma_{j,l=1}^{n-1}X_jY_l\xi_j\xi_lh'(0)\left[8i\frac{5\xi_n^2-1}{(\xi_n-i)^5(\xi_n+i)^4}
 +\frac{4(5\xi_n^2-1)+12i(\xi_n^3-\xi_n)}{(\xi_n-i)^6(\xi_n+i)^4}\right]\nonumber\\
&+X_nY_nh'(0)\left[\frac{(4i-4)(\xi^2-1)+48(\xi_n^3-\xi_n)}{(\xi_n-i)^5(\xi_n+i)^4}
 -\frac{4(5\xi_n^2-1)+12i(\xi_n^3-\xi_n)}{(\xi_n-i)^6(\xi_n+i)^4}\right]\nonumber\\
&+8\Sigma_{j=1}^{n-1}X_jY_n\xi_j\left[\frac{(6-3ih'(0))(\xi_n^3-\xi_n)-2i(5\xi_n^2-1)}{(\xi_n-i)^5(\xi_n+i)^4}
 +\frac{2(5\xi_n^2-1)+6i(\xi_n^3-\xi_n)}{(\xi_n-i)^6(\xi_n+i)^4}\right]\nonumber\\
&+8\Sigma_{l=1}^{n-1}X_nY_l\xi_l\left[\frac{(6-3ih'(0))(\xi_n^3-\xi_n)-2i(5\xi_n^2-1)}{(\xi_n-i)^5(\xi_n+i)^4}
 +\frac{2(5\xi_n^2-1)+6i(\xi_n^3-\xi_n)}{(\xi_n-i)^6(\xi_n+i)^4}\right].
\end{align}

Therefore, we get
\begin{align}\label{35}
\widetilde{\Phi}_2&=\frac{1}{2}\int_{|\xi'|=1}\int^{+\infty}_{-\infty}\bigg\{
\sum_{j,l=1}^{n-1}X_jY_l\xi_j\xi_lh'(0)\left[8i\frac{5\xi_n^2-1}{(\xi_n-i)^5(\xi_n+i)^4}
 +\frac{4(5\xi_n^2-1)+12i(\xi_n^3-\xi_n)}{(\xi_n-i)^6(\xi_n+i)^4}\right]\nonumber\\
&+X_nY_nh'(0)\left[\frac{(4i-4)(\xi^2-1)+48(\xi_n^3-\xi_n)}{(\xi_n-i)^5(\xi_n+i)^4}
 -\frac{4(5\xi_n^2-1)+12i(\xi_n^3-\xi_n)}{(\xi_n-i)^6(\xi_n+i)^4}\right]
 \bigg\}d\xi_n\sigma(\xi')dx'\nonumber\\
 &=\sum_{j,l=1}^{n-1}X_jY_lh'(0)\Omega_3\int_{\Gamma^{+}}\left[8i\frac{5\xi_n^2-1}{(\xi_n-i)^5(\xi_n+i)^4}
 +\frac{4(5\xi_n^2-1)+12i(\xi_n^3-\xi_n)}{(\xi_n-i)^6(\xi_n+i)^4}\right]\xi_j\xi_ld\xi_{n}dx'\nonumber\\
 &+X_nY_nh'(0)\Omega_3\int_{\Gamma^{+}}\left[\frac{(4i-4)(\xi^2-1)+48(\xi_n^3-\xi_n)}{(\xi_n-i)^5(\xi_n+i)^4}
 -\frac{4(5\xi_n^2-1)+12i(\xi_n^3-\xi_n)}{(\xi_n-i)^6(\xi_n+i)^4}\right]d\xi_{n}dx'\nonumber\\
&=- \frac{592}{3}\pi^{2}\sum_{j=1}^{n-1}X_jY_jh'(0)  dx'
-\left(\frac{461}{4}+\frac{23}{4}i\right)X_nY_n h'(0)\pi\Omega_3dx',
\end{align}
where ${\rm \Omega_{3}}$ is the canonical volume of $S^{2}.$

 {\bf case a)~III)}~$r=1,~l=-3,~j=|\alpha|=0,~k=1$.

 By (3.12), we get
\begin{align}\label{36}
\widetilde{\Phi}_3&=-\frac{1}{2}\int_{|\xi'|=1}\int^{+\infty}_{-\infty}
\mathrm{Tr} [\partial_{\xi_n}\pi^+_{\xi_n}\sigma_{1}(\widetilde{\nabla}_{X}\widetilde{\nabla}_{Y}D_{T}^{-1})\times
\partial_{\xi_n}\partial_{x_n}\sigma_{-3}((D_{T}^{*}D_{T}D_{T}^{*})^{-1})](x_0)d\xi_n\sigma(\xi')dx'\nonumber\\
&=\frac{1}{2}\int_{|\xi'|=1}\int^{+\infty}_{-\infty}
\mathrm{Tr} [\partial_{\xi_n}^2\pi^+_{\xi_n}\sigma_{1}(\widetilde{\nabla}_{X}\widetilde{\nabla}_{Y}D_{T}^{-1})\times
\partial_{x_n}\sigma_{-3}((D_{T}^{*}D_{T}D_{T}^{*})^{-1})](x_0)d\xi_n\sigma(\xi')dx'.
\end{align}
\noindent By Lemma 5.3, we have
\begin{eqnarray}\label{37}
\partial_{x_n}\sigma_{-3}((D_{T}^{*}D_{T}D_{T}^{*})^{-1})(x_0)|_{|\xi'|=1}
=\frac{\sqrt{-1}\partial_{x_n}[c(\xi')]}{(1+\xi_n^2)^4}-\frac{2\sqrt{-1}h'(0)c(\xi)|\xi'|^2_{g^{\partial M}}}{(1+\xi_n^2)^6}.
\end{eqnarray}
By integrating formula we obtain
\begin{align}\label{38}
\pi^+_{\xi_n}\sigma_{1}(\widetilde{\nabla}_{X}\widetilde{\nabla}_{Y}D_{T}^{-1})
&=-\frac{c(\xi')+ic(\mathrm{d}x_n)}{2(\xi_n-i)}\Sigma_{j,l=1}^{n-1}X_jY_l\xi_j\xi_l
-\frac{c(\xi')+ic(\mathrm{d}x_n)}{2(\xi_n-i)}X_nY_n\nonumber\\
&-\frac{ic(\xi')-c(\mathrm{d}x_n)}{2(\xi_n-i)}\Sigma_{j=1}^{n-1}X_jY_n\xi_j
-\frac{ic(\xi')-c(\mathrm{d}x_n)}{2(\xi_n-i)}\Sigma_{l=1}^{n-1}X_nY_l\xi_l.
\end{align}
Then, we have
\begin{align}\label{mmmmm}
\partial_{\xi_n}^2\pi^+_{\xi_n}\sigma_{1}(\widetilde{\nabla}_{X}\widetilde{\nabla}_{Y}D_{T}^{-1})
=-\frac{c(\xi')+ic(\mathrm{d}x_n)}{(\xi_n-i)^3}\Sigma_{j,l=1}^{n-1}X_jY_l\xi_j\xi_l
-\frac{c(\xi')+ic(\mathrm{d}x_n)}{(\xi_n-i)^3}X_nY_n.
\end{align}

We note that $i<n,~\int_{|\xi'|=1}\xi_{i_{1}}\xi_{i_{2}}\cdots\xi_{i_{2d+1}}\sigma(\xi')=0$,
so we omit some items that have no contribution for computing {\bf case a)~III)}, then
\begin{align}\label{39}
&\mathrm{Tr} [\partial_{\xi_n}\pi^+_{\xi_n}\sigma_{1}(\widetilde{\nabla}_{X}\widetilde{\nabla}_{Y}D_{T}^{-1})\times
\partial_{\xi_n}\partial_{x_n}\sigma_{-3}((D_{T}^{*}D_{T}D_{T}^{*})^{-1})](x_0)\nonumber\\
&=-2\frac{h'(0)}{(\xi_n-i)^5(\xi_n+i)^2}\Sigma_{j,l=1}^{n-1}X_jY_l\xi_j\xi_l
-2\frac{h'(0)}{(\xi_n-i)^5(\xi_n+i)^2}X_nY_n.\nonumber\\
\end{align}

Therefore, we get
\begin{align}\label{41}
\widetilde{\Phi}_3&=\frac{1}{2}\int_{|\xi'|=1}\int^{+\infty}_{-\infty}
\bigg(-2\frac{h'(0)}{(\xi_n-i)^5(\xi_n+i)^2}\Sigma_{j,l=1}^{n-1}X_jY_l\xi_j\xi_l
-2\frac{h'(0)}{(\xi_n-i)^5(\xi_n+i)^2}X_nY_n\bigg)d\xi_n\sigma(\xi')dx'\nonumber\\
&=-\sum_{j,l=1}^{n-1}X_jY_lh'(0)\Omega_3\int_{\Gamma^{+}}\frac{1}{(\xi_n-i)^5(\xi_n+i)^2}\xi_j\xi_ld\xi_{n}dx'
-X_nY_nh'(0)\Omega_3\int_{\Gamma^{+}}\frac{1}{(\xi_n-i)^5(\xi_n+i)^2}d\xi_{n}dx'\nonumber\\
&=-\sum_{j,l=1}^{n-1}X_jY_lh'(0)\Omega_3\frac{2\pi i}{4!}
\left[\frac{1}{(\xi_n+i)^2}\right]^{(4)}\bigg|_{\xi_n=i}dx'
+2X_nY_nh'(0)\Omega_3\frac{2\pi i}{4!}
\left[\frac{1}{(\xi_n+i)^2}\right]^{(4)}\bigg|_{\xi_n=i}dx'\nonumber\\
&= \frac{5 i\pi^{2}}{6}\sum_{j=1}^{n-1}X_jY_jh'(0) dx'+\frac{5i}{8}X_nY_n h'(0)\pi\Omega_3dx'.
\end{align}

 {\bf case b)}~$r=0,~l=-3,~k=j=|\alpha|=0$.

By (3.12), we get
\begin{align}\label{42}
\widetilde{\Phi}_4&=-i\int_{|\xi'|=1}\int^{+\infty}_{-\infty}
\mathrm{Tr}[\pi^+_{\xi_n}\sigma_{0}(\widetilde{\nabla}_{X}\widetilde{\nabla}_{Y}D_{T}^{-1})\times
\partial_{\xi_n}\sigma_{-3}((D_{T}^{*}D_{T}D_{T}^{*})^{-1})](x_0)d\xi_n\sigma(\xi')dx'.
\end{align}
By Lemma 5.3, we obtain
\begin{align}\label{43}
\partial_{\xi_n}\sigma_{-3}((D_{T}^{*}D_{T}D_{T}^{*})^{-1})(x_0)|_{|\xi'|=1}
=\frac{ic(\mathrm{d}x_n)}{(1+\xi_n^2)^2}-\frac{4\sqrt{-1}\xi_nc(\xi)}{(1+\xi_n^2)^3}.
\end{align}
By Lemma 5.2, we have
\begin{align}
\sigma_{0}(\widetilde{\nabla}_{X}\widetilde{\nabla}_{Y}D_{T}^{-1})=&
\sigma_{2}(\widetilde{\nabla}_{X}\widetilde{\nabla}_{Y})\sigma_{-2}(D_{T}^{-1})
+\sigma_{1}(\widetilde{\nabla}_{X}\widetilde{\nabla}_{Y})\sigma_{-1}(D_{T}^{-1})\nonumber\\
&+\sum_{j=1}^{n}\partial _{\xi_{j}}\big[\sigma_{2}(\widetilde{\nabla}_{X}\widetilde{\nabla}_{Y})\big]
D_{x_{j}}\big[\sigma_{-1}(D_{T}^{-1})\big].
\end{align}

(1) Explicit representation the first item of (5.28)
\begin{align}
&\sigma_{2}(\widetilde{\nabla}_{X}\widetilde{\nabla}_{Y})\sigma_{-2}(D_{T}^{-1})(x_0)|_{|\xi'|=1}\nonumber\\
=&-\sum_{j,l=1}^{n}X_jY_l\xi_j\xi_l\left[\frac{c(\xi)\sigma_0(D_{T})c(\xi)}{|\xi|^4}
+\frac{c(\xi)}{|\xi|^6}\Sigma_jc(\mathrm{d}x_j)(\partial_{x_j}[c(\xi)]|\xi|^2-c(\xi)\partial_{x_j}(|\xi|^2))\right],
\end{align}
By integrating formula, we obtain
\begin{align}
\pi^+_{\xi_n}\left[\frac{c(\xi)\sigma_0(D_{T})(x_0)c(\xi)+c(\xi)c(dx_n)\partial_{x_n}[c(\xi')](x_0)}{(1+\xi_n^2)^2}\right]
=-\frac{A_1}{4(\xi_n-i)}-\frac{A_2}{4(\xi_n-i)^2}+\frac{A_3}{4(\xi_n-i)^2},
\end{align}
where
\begin{align}
A_1=&ic(\xi')p_0c(\xi')+ic(dx_n)(-\frac{3}{4}h'(0)c(dx_n))c(dx_n)+ic(\xi')c(dx_n)\partial_{x_n}[c(\xi')],\\
A_2=&[c(\xi')+ic(dx_n)]p_0[c(\xi')+ic(dx_n)]+c(\xi')c(dx_n)\partial_{x_n}c(\xi')-i\partial_{x_n}[c(\xi')],\\
A_3=&\sum_{j,l=1}^{n-1}X_{j}Y_l\xi_{j}\xi_{l}\Big( (-2-i\xi_{n})c(\xi')(u+v)c(\xi')-ic(dx_n)(u+v)c(\xi')-i c(\xi')(u+v)c(dx_n)\nonumber\\
   &-i\xi_{n}c(dx_n)(u+v)c(dx_n)\Big)+X_{n}Y_n\Big( -i\xi_{n}c(\xi')(u+v)c(\xi')-ic(dx_n)(u+v)c(\xi')\nonumber\\
   &-i c(\xi')(u+v)c(dx_n)+c(dx_n)(u+v)c(dx_n)\Big).
\end{align}
In the same way we get
\begin{align}
\pi^+_{\xi_n}\left[\frac{c(\xi)c(dx_n)c(\xi)}{(1+\xi_n)^3}(x_0)|_{|\xi'|=1}\right]
&=\frac{1}{2}\left[\frac{c(dx_n)}{4i(\xi_n-i)}+\frac{c(dx_n)-ic(\xi')}{8(\xi_n-i)^2}
+\frac{3\xi_n-7i}{8(\xi_n-i)^3}[ic(\xi')-c(dx_n)]\right].
\end{align}
Then adding these identities gives
\begin{align}\label{39}
&\mathrm{Tr} \Big(\pi^+_{\xi_n}\Big(\sum_{j=1}^{n-1}\sqrt{-1} T(X,\cdot,\cdot)Y_j\xi_j+\sum_{j=1}\sqrt{-1} T(Y,\cdot,\cdot)X_j\xi_j\Big)\times
\partial_{\xi_n}\sigma_{-3}((D_{T}^{*}D_{T}D_{T}^{*})^{-1})\Big)(x_0)\nonumber\\
=&\sum_{j=1}^{n-1}Y_j\xi_j \frac{2i\xi_n}{(\xi_n-i)(1+\xi_n^{2})^{3}}{\rm trace}\Big( T(X,\cdot,\cdot)c(dx_n)c(\xi')\Big)\nonumber\\
&+\sum_{j=1}^{n-1}Y_j\xi_j \frac{3\xi_n^{2}-1}{2(\xi_n-i)(1+\xi_n^{2})^{3}}{\rm trace}\Big( T(X,\cdot,\cdot)c(\xi')c(dx_n)\Big)\nonumber\\
&+\sum_{j=1}^{n-1}X_j\xi_j \frac{2i\xi_n}{(\xi_n-i)(1+\xi_n^{2})^{3}}{\rm trace}\Big( T(Y,\cdot,\cdot)c(dx_n)c(\xi')\Big)\nonumber\\
&+\sum_{j=1}^{n-1}Y_j\xi_j \frac{3\xi_n^{2}-1}{(\xi_n-i)(1+\xi_n^{2})^{3}}{\rm trace}\Big( T(Y,\cdot,\cdot)c(dx_n)c(\xi')\Big)\nonumber\\
=&\sum_{j=1}^{n-1}Y_j\xi_j \frac{2i\xi_n}{(\xi_n-i)(1+\xi_n^{2})^{3}} \sum_{i=1}^{n-1}T(X,e_{i},e_{n})\xi_i
-\sum_{j=1}^{n-1}Y_j\xi_j \frac{3\xi_n^{2}-1}{2(\xi_n-i)(1+\xi_n^{2})^{3}}\sum_{i=1}^{n-1}T(X,e_{i},e_{n})\xi_i \nonumber\\
&+\sum_{j=1}^{n-1}X_j\xi_j \frac{2i\xi_n}{(\xi_n-i)(1+\xi_n^{2})^{3}}\sum_{i=1}^{n-1}T(Y,e_{i},e_{n})\xi_i
-\sum_j^{n-1}Y_j\xi_j \frac{3\xi_n^{2}-1}{(\xi_n-i)(1+\xi_n^{2})^{3}}\sum_{i=1}^{n-1}T(Y,e_{i},e_{n})\xi_i
\end{align}
Substituting (5.35) into (5.26) yields
\begin{align}\label{39}
&-i\int_{|\xi'|=1}\int^{+\infty}_{-\infty}
\mathrm{Tr} [\pi^+_{\xi_n}\Big(\sigma_{2}(\widetilde{\nabla}_{X}\widetilde{\nabla}_{Y})\frac{c(\xi)(u+v)c(\xi)}{(1+\xi_n)^3}\Big)\times
\partial_{\xi_n}\sigma_{-3}((D_{T}^{*}D_{T}D_{T}^{*})^{-1})](x_0)d\xi_n\sigma(\xi')dx'\nonumber\\
=&\Big(\sum_{j,l=1}^{n-1}X_{j}Y_l\frac{2\pi}{3}+\frac{3}{8}X_{n}Y_n\Big)\sum_{i=1}^{n-1}A_{iin}\pi\Omega_3dx'
+\sum_{j=1}^{n-1}\Big(X_{j}\mathrm{Tr}\big(T(X,e_{j},e_{n})\big)Y_n
+Y_j\mathrm{Tr}\big(T(Y,e_{j},e_{n})\big)\Big)\frac{4i\pi^{2}}{3}\Omega_3dx'\nonumber\\
=& \sum_{j,l=1}^{n-1}X_{j}Y_l\frac{2\pi^{2}}{3}\sum_{i=1}^{n-1}A_{iin} dx'+\frac{3}{8}X_{n}Y_n \sum_{i=1}^{n-1}A_{iin}\pi\Omega_3dx'.
\end{align}

(2) Explicit representation the second item of (5.28)
\begin{align}
&\sigma_{1}(\widetilde{\nabla}_{X}\widetilde{\nabla}_{Y})\sigma_{-1}(D_{T}^{-1})(x_0)|_{|\xi'|=1}\nonumber\\
=&\Big(\sqrt{-1}\sum_{j,l=1}^nX_j\frac{\partial_{Y_l}}{\partial_{x_j}}\partial_{x_l}
+\sqrt{-1}\sum_jA(Y)X_j\xi_j+\sqrt{-1}\sum_lA(Y)Y_l\xi_l\nonumber\\
&+\sum_j\sqrt{-1} T(X,\cdot)Y_j\xi_j+\sum_j\sqrt{-1} T(Y,\cdot)X_j\xi_j \Big)\frac{\sqrt{-1}c(\xi)}{|\xi|^{2}};
\end{align}
By integrating formula we get
\begin{align}
&\pi^+_{\xi_n}\left(\Big(\sum_j^{n}\sqrt{-1} T(X,\cdot)Y_j\xi_j+\sum_j^{n}\sqrt{-1} T(Y,\cdot)X_j\xi_j \Big)
\frac{\sqrt{-1}c(\xi)}{|\xi|^{2}}\right)\nonumber\\
=&\pi^+_{\xi_n}\left(\Big(\sum_j^{n-1}\sqrt{-1} T(X,\cdot)Y_j\xi_j+\sum_j^{n-1}\sqrt{-1} T(Y,\cdot)X_j\xi_j \Big)
\frac{\sqrt{-1}c(\xi)}{|\xi|^{2}}\right)\nonumber\\
&+\pi^+_{\xi_n}\left(\Big( \sqrt{-1} T(X,\cdot)Y_n\xi_n+\sqrt{-1} T(Y,\cdot)X_n\xi_n \Big)
\frac{\sqrt{-1}c(\xi)}{|\xi|^{2}}\right)\nonumber\\
=&\Big(\sum_j^{n-1}\sqrt{-1} T(X,\cdot)Y_j\xi_j+\sum_j^{n-1}\sqrt{-1} T(Y,\cdot)X_j\xi_j \Big)
\frac{ic(\xi')-c(dx_n)}{2(\xi_{n}-i)}\nonumber\\
&+\Big( \sqrt{-1} T(X,\cdot)Y_n\xi_n+\sqrt{-1} T(Y,\cdot)X_n\xi_n \Big)
\frac{-c(\xi')-ic(dx_n)}{2(\xi_{n}-i)}
\end{align}
We note that $i<n,~\int_{|\xi'|=1}\xi_{i_{1}}\xi_{i_{2}}\cdots\xi_{i_{2d+1}}\sigma(\xi')=0$,
and $\sum_{i\neq s\neq t}Tr[A_{ist}c(\widetilde{e_{i}})c(\widetilde{e_{s}})c(\widetilde{e_{t}})c(\xi')]=0$,
then
\begin{align}\label{39}
\mathrm{Tr} \left(\pi^+_{\xi_n}\Big(\sigma_{1}(\widetilde{\nabla}_{X}\widetilde{\nabla}_{Y})\sigma_{-1}(D_{T}^{-1})\Big)\times
\partial_{\xi_n}\sigma_{-3}((D_{T}^{*}D_{T}D_{T}^{*})^{-1})\right)(x_0)=0.
\end{align}

(3) Explicit representation the third item of (5.28)
\begin{align}
\sum_{j=1}^{n}\sum_{\alpha}\frac{1}{\alpha!}\partial^{\alpha}_{\xi}\big[\sigma_{2}(\widetilde{\nabla}_{X}\widetilde{\nabla}_{Y})\big]
D_x^{\alpha}\big[\sigma_{-1}(D_{T}^{-1})\big](x_0)|_{|\xi'|=1}
=&\sum_{j=1}^{n}\partial_{\xi_{j}}
\big[\sigma_{2}(\widetilde{\nabla}_{X}\widetilde{\nabla}_{Y})\big]
(-\sqrt{-1})\partial_{x_{j}}\big[\sigma_{-1}( D_{T} ^{-1})\big]\nonumber\\
=&\sum_{j=1}^{n}\partial_{\xi_{j}}\big[-\sum_{j,l=1}^nX_jY_l\xi_j\xi_l\big]
(-\sqrt{-1})\partial_{x_{j}}\big[\frac{\sqrt{-1}c(\xi)}{|\xi|^{2}}\big]\nonumber\\
=&\sum_{j=1}^{n}\sum_{l=1}^{n}\sqrt{-1}(x_{j}Y_l+x_{l}Y_j)\xi_{l}\partial_{x_{j}}(\frac{\sqrt{-1}c(\xi)}{|\xi|^{2}}).
\end{align}
By integrating formula we obtain
\begin{align}
&\pi^+_{\xi_n}\left(\sum_{j=1}^{n}\sum_{\alpha}\frac{1}{\alpha!}\partial^{\alpha}_{\xi}\big[\sigma_{2}(\widetilde{\nabla}_{X}
\widetilde{\nabla}_{Y})\big]
D_x^{\alpha}\big[\sigma_{-1}(D_{T}^{-1})\big]\right)\nonumber\\
=&\pi^+_{\xi_n}\left(\sum_{l=1}^{n-1}\sqrt{-1}(x_{n}Y_l+x_{l}Y_n)\xi_{l}\partial_{x_{n}}(\frac{\sqrt{-1}c(\xi)}{|\xi|^{2}})
\right)+\pi^+_{\xi_n}\left(\sqrt{-1}(x_{n}Y_n+x_{n}Y_n)\xi_{n}\partial_{x_{n}}(\frac{\sqrt{-1}c(\xi)}{|\xi|^{2}})
\right)\nonumber\\
=&\sum_{l=1}^{n-1}(x_{n}Y_l+x_{l}Y_n)\xi_{l}\Big(\frac{i\partial_{x_{n}}(c(\xi'))}{2(\xi_n-i)}
+h'(0)\frac{(-2-i\xi_n )c(\xi')}{4(\xi_n-i)^2}
-h'(0)\frac{ic(dx_n)}{4(\xi_n-i)^2}\Big)\nonumber\\
&+x_{n}Y_n\Big(\frac{-\partial_{x_{n}}(c(\xi'))}{(\xi_n-i)}
+h'(0)\frac{(-i )c(\xi')}{2(\xi_n-i)^2}
-h'(0)\frac{-i\xi_n c(dx_n)}{2(\xi_n-i)^2}\Big)
\end{align}
Substituting (5.41) into (5.26) yields
\begin{align}\label{39}
&-i\int_{|\xi'|=1}\int^{+\infty}_{-\infty}
\mathrm{Tr}\Big[\pi^+_{\xi_n}\Big(\sum_{j=1}^{n}\sum_{\alpha}\frac{1}{\alpha!}\partial^{\alpha}_{\xi}
\big[\sigma_{2}(\widetilde{\nabla}_{X}\widetilde{\nabla}_{Y})\big]
D_x^{\alpha}\big[\sigma_{-1}(D_{T}^{-1})\big]\Big)\nonumber\\
&\times
\partial_{\xi_n}\sigma_{-3}((D_{T}^{*}D_{T}D_{T}^{*})^{-1})\Big](x_0)d\xi_n\sigma(\xi')dx'\nonumber\\
&=\frac{7-15i}{8}X_{n}Y_n\pi h'(0)\Omega_3dx'.
\end{align}
Summing up (1), (2) and (3) leads to the desired equality
\begin{align}\label{41}
\widetilde{\Phi}_4
&=\Big(\frac{55\pi^{2}}{3}\sum_{j=1}^{n-1}X_jY_j+\frac{4-15i}{8}X_nY_n\pi\Omega_3\Big)h'(0)dx'
+\Big(\frac{2\pi^{2}}{3}\sum_{j,l=1}^{n-1}X_{j}Y_l+\frac{3}{8}X_{n}Y_n\pi\Omega_3\Big)\sum_{i=1}^{n-1}A_{iin}dx'.
\end{align}

 {\bf  case c)}~$r=1,~\ell=-4,~k=j=|\alpha|=0$.

By (3.12), we get
\begin{align}\label{61}
\widetilde{\Phi}_5&=-\int_{|\xi'|=1}\int^{+\infty}_{-\infty}\mathrm{Tr} [\pi^+_{\xi_n}
\sigma_{1}(\widetilde{\nabla}_{X}\widetilde{\nabla}_{Y}D_{T}^{-1})\times
\partial_{\xi_n}\sigma_{-4}(D_{T}^{*}D_{T}D_{T}^{*})^{-1})](x_0)d\xi_n\sigma(\xi')dx'\nonumber\\
&=\int_{|\xi'|=1}\int^{+\infty}_{-\infty}\mathrm{Tr}
[\partial_{\xi_n}\pi^+_{\xi_n}\sigma_{1}(\widetilde{\nabla}_{X}\widetilde{\nabla}_{Y}D_{T}^{-1})\times
\sigma_{-4}(D_{T}^{*}D_{T}D_{T}^{*})^{-1})](x_0)d\xi_n\sigma(\xi')dx'.
\end{align}
By Lemma 5.3, we have
\begin{align}\label{62}
\sigma_{-4}(D_{T}^{*}D_{T}D_{T}^{*})^{-1})(x_0)|_{|\xi'|=1}=&\frac{1}{(\xi_n^2+1)^4}
\left[\left(\frac{11}{2}\xi_n(1+\xi_n^2)+8i\xi_n\right)h'(0)c(\xi')\right.\nonumber\\
&+\left[-2i+6i\xi_n^2-\frac{7}{4}(1+\xi_n^2)
 +\frac{15}{4}\xi_n^2(1+\xi^2_n)\right]h'(0)c(\mathrm{d}x_n) \nonumber\\
&\left.-3i\xi_n(1+\xi^2_n)\partial_{x_n}c(\xi')
 +i(1+\xi^2_n)c(\xi')c(\mathrm{d}x_n)\partial_{x_n}c(\xi')\right] \nonumber\\
 &+\frac{c(\xi)(3u-v)|\xi|^2c(\xi)}{|\xi|^8},
\end{align}
and
\begin{align}\label{621}
\partial_{\xi_n}\pi^+_{\xi_n}\sigma_{1}(\widetilde{\nabla}_{X}\widetilde{\nabla}_{Y}D_{T}^{-1})(x_0)|_{|\xi'|=1}
&=\frac{c(\xi')+ic(\mathrm{d}x_n)}{2(\xi_n-i)^2}\Sigma_{j,l=1}^{n-1}X_jY_l\xi_j\xi_l
-\frac{c(\xi')+ic(\mathrm{d}x_n)}{2(\xi_n-i)^2}X_nY_n \nonumber\\
&+\frac{ic(\xi')-c(\mathrm{d}x_n)}{2(\xi_n-i)^2}\Sigma_{j=1}^{n}X_jY_n\xi_j
+\frac{ic(\xi')-c(\mathrm{d}x_n)}{2(\xi_n-i)^2}\Sigma_{l=1}^{n}X_nY_l\xi_l.
\end{align}

We note that $i<n,~\int_{|\xi'|=1}\xi_{i_{1}}\xi_{i_{2}}\cdots\xi_{i_{2d+1}}\sigma(\xi')=0$,
so we omit some items that have no contribution for computing {\bf case c)}. Here
\begin{align}\label{63}
{\rm tr}[c(\xi')c(\xi')c(\mathrm{d}x_n)\partial_{x_n}c(\xi')]=0 ;\nonumber\\
{\rm tr}[c(\mathrm{d}x_n)c(\xi')c(\mathrm{d}x_n)\partial_{x_n}c(\xi')]=-2h'(0).
\end{align}
Also, straightforward computations yield
\begin{align}\label{71}
&{\rm tr}\bigg[\partial_{\xi_n}\pi^+_{\xi_n}\sigma_{-1}(\widetilde{\nabla}_{X}\widetilde{\nabla}_{Y}D_{T}^{-1})\times
\frac{1}{(\xi_n^2+1)^4}
\Big(\big(\frac{11}{2}\xi_n(1+\xi_n^2)+8i\xi_n\big)h'(0)c(\xi')\nonumber\\
&+\big(-2i+6i\xi_n^2-\frac{7}{4}(1+\xi_n^2)
 +\frac{15}{4}\xi_n^2(1+\xi^2_n)\big)h'(0)c(\mathrm{d}x_n) \nonumber\\
&-3i\xi_n(1+\xi^2_n)\partial_{x_n}c(\xi')
 +i(1+\xi^2_n)c(\xi')c(\mathrm{d}x_n)\partial_{x_n}c(\xi')\Big) \bigg](x_0)|_{|\xi'|=1}\nonumber\\
=&\sum_{j,l=1}^{n-1}X_jY_l\xi_j\xi_l\frac{h'(0)(7+6i-(20-15i)\xi_n-(7-6i)\xi_n^2+15i\xi_n^3)}{(\xi_n-i)^5(\xi_n+i)^4}\nonumber\\
&+X_nY_n\frac{(3i-11)\xi_n(1-\xi_n^2)-16i\xi_n+(13+\frac{7}{2}i)(1+\xi_n^2)-16-\frac{15}{2}\xi_n^2(1+\xi_n^2)}{(\xi_n-i)^2(\xi_n+i)^4},
\end{align}
and
\begin{align}\label{71}
&{\rm tr}\Big(\partial_{\xi_n}\pi^+_{\xi_n}\sigma_{-1}(\widetilde{\nabla}_{X}\widetilde{\nabla}_{Y}D_{T}^{-1})\times
\frac{c(\xi)(3u-v)|\xi|^2c(\xi)}{|\xi|^8}\Big)(x_0)|_{|\xi'|=1}\nonumber\\
&=\sum_{j,l=1}^{n-1}X_jY_l\xi_j\xi_l\frac{-3i\pi}{8}\sum_{i=1}^{n}A_{iin}
+X_nY_n\frac{-3i\pi}{8}\sum_{i=1}^{n}A_{iin}.
\end{align}
From (5.44),(5.48) and (5.49), we get
\begin{align}\label{74}
\widetilde{\Phi}_5
=&\Big(\big(-\frac{35}{3}+\frac{50}{3}i\big)\sum_{j=1}^{n-1}X_jY_j
\pi^{2}+\big(5-\frac{137}{32}i\big)X_nY_n\pi\Omega_3\Big)h'(0)dx'.
\end{align}

Let $X=X^T+X_n\partial_n,~Y=Y^T+Y_n\partial_n,$ then we have $\sum_{j=1}^{n-1}X_jY_j=g(X^T,Y^T).$
 Now $\Phi$ is the sum of the cases (a), (b) and (c). Combining with the five cases, this yields
\begin{align}\label{795}
\widetilde{\Phi}=\sum_{i=1}^5\widetilde{\Phi}_i
=&\left[\left(-\frac{2801}{12}-\frac{33i}{32}\right)X_nY_n\pi\Omega_3+\left(-\frac{572}{3}+\frac{35i}{2}\right)
\pi^{2} g(X^T,Y^T)\right]h'(0)dx'\nonumber\\
&+\Big(\frac{3}{8}-\frac{3i}{8})X_nY_n\pi\Omega_3+\frac{4+3i}{6}\pi^{2} g(X^T,Y^T)\Big)\sum_{i=1}^{n}A_{iin}dx'.
\end{align}
So, we are reduced to prove the following.
\begin{thm}\label{thmb1}
Let $M$ be a 4-dimensional compact manifold with  boundary $\partial M$ and $\widetilde{\nabla}$ be an orthogonal
connection with torsion. Then we get the noncommutative residue associated to $\widetilde{\nabla}_{X}\widetilde{\nabla}_{Y}D_{T}^{-1}$
and $(D_{T}^{*}D_{T}D_{T}^{*})^{-1}$ on compact manifolds with boundary
\begin{align}
\label{b2773}
&\widetilde{{\rm Wres}}[\pi^+(\widetilde{\nabla}_{X}\widetilde{\nabla}_{Y}D_{T}^{-1})
\circ\pi^+((D_{T}^{*}D_{T}D_{T}^{*})^{-1})]\nonumber\\
=&\frac{4\pi^{2}}{3}\int_{M}\Big(Ric(X,Y)-\frac{1}{2}sg(X,Y)\Big) vol_{g}\nonumber\\
&+\int_{M}
\Big( -\frac{1}{2}R^{g}-3\mathrm{div}^{g}(X)+3\parallel T\parallel^{2}+9\parallel X\parallel^{2}\Big)g(X,Y) vol_{g}\nonumber\\
&+\int_{\partial M}
\bigg(\Big((\frac{-2801}{24}-\frac{33i}{32})X_nY_n\pi\Omega_3+(\frac{35i}{2}-\frac{572}{3})\pi^{2 } g(X^T,Y^T)\Big )  h'(0)\nonumber\\
&+\Big((\frac{3}{8}-\frac{3i}{8})X_nY_n\pi\Omega_3+\frac{4+3i}{6}\pi^{2} g(X^T,Y^T)\Big)\sum_{i=1}^{n} A_{iin}\bigg)
vol_{\partial M},
\end{align}
where $R_{g}$ denotes the curvature tensor and $s$ is the scalar curvature.
\end{thm}

\section*{ Acknowledgements}
The first author was supported by NSFC. 11501414. The second author was supported by NSFC. 11771070.
 The authors also thank the referee for his (or her) careful reading and helpful comments.

\end{document}